\documentclass[11pt,a4paper]{amsart}[2000]
\title{Rokhlin dimension and $C^*$-dynamics}

\date{\today}

\author{Ilan Hirshberg}
\address{Department of Mathematics, Ben Gurion University of the Negev, P.O.B. 653, Be'er  Sheva 84105, Israel}
\email{ilan@math.bgu.ac.il}

\author{Wilhelm Winter}
\address{Mathematisches Institut,
Westfalische Wilhelms-Universit{\"a}t,  48149 M{\"u}nster, Germany}
\email{wwinter@uni-muenster.de}

\author{Joachim Zacharias}

\address{School of Mathematics and Statistics, University of Glasgow, University Gardens, Glasgow Q12 8QW, Scotland}
\email{Joachim.Zacharias@glasgow.ac.uk}

\usepackage{amssymb}
\usepackage{amsthm}
\usepackage{xypic}
\usepackage{amsmath}

\theoremstyle{plain}

\newtheorem{Thm}{\sc Theorem}[section]
\newtheorem{Cor}[Thm]{\sc Corollary}
\newtheorem{Lemma}[Thm]{\sc Lemma}

\newtheorem{Prop}[Thm]{\sc Proposition}

\theoremstyle{definition}
\newtheorem{Def}[Thm]{\sc Definition}
\newtheorem{Not}[Thm]{\sc Notation}

\newtheorem{Rmk}[Thm]{\sc Remark}
\newtheorem{Rmks}[Thm]{\sc Remarks}

\renewcommand{\epsilon}{\varepsilon}

\newcommand{\B}{\mathcal{B}}
\newcommand{\A}{\mathcal{A}}

\newcommand{\D}{\mathcal{D}}

\newcommand{\Zh}{\mathcal{Z}}

\newcommand{\E}{\mathcal{E}}
\newcommand{\F}{F}

\newcommand{\T}{{\mathbb T}}
\newcommand{\R}{{\mathbb R}}
\newcommand{\N}{{\mathbb N}}
\newcommand{\Z}{{\mathbb Z}}
\newcommand{\C}{{\mathbb C}}

\newcommand{\aut}{\mathrm{Aut}}

\newcommand{\eps}{\varepsilon}

\newcommand{\dr}{\textup{dr}}
\newcommand{\nd}{\dim_{\mathrm{nuc}}}

\newcommand{\dimrokh}{\dim_{\mathrm{Rok}}}
\newcommand{\dimrokhct}{\dim^{\mathrm{c}}_{\mathrm{Rok}}}
\newcommand{\dimrokhbar}{\,\overline{\dim}_{\mathrm{Rok}}}
\newcommand{\dimrokhsingle}{\dim^{\mathrm{s}}_{\mathrm{Rok}}}
\newcommand{\dimrokhbarct}{\,\overline{\dim}^{\mathrm{c}}_{\mathrm{Rok}}}
\newcommand{\dimrokhbarsingle}{\,\overline{\dim}^{\mathrm{s}}_{\mathrm{Rok}}}

\thanks{Research partially supported by the Israel Science Foundation (grant No.\ 
1471/07), GIF (grant No.\ 1137-30.6/2011), by EPSRC (grants No.\  EP/G014019/1 and No.\ EP/I019227/1), by the DFG (SFB 878) and by the CRM Barcelona.}

\begin{document}
\begin{abstract}
We develop the concept of Rokhlin dimension for  integer and for finite group actions  on $C^{*}$-algebras. Our notion generalizes the so-called Rokhlin property, which can be thought of as Rokhlin dimension 0. We show that finite Rokhlin dimension is prevalent and appears in cases in which the Rokhlin property cannot be expected: the property of having finite Rokhlin dimension is generic for automorphisms of $\Zh$-stable $C^{*}$-algebras, where $\Zh$ denotes the Jiang-Su algebra. Moreover, crossed products by automorphisms with finite Rokhlin dimension preserve the property of having finite nuclear dimension, and under a mild additional hypothesis also preserve $\Zh$-stability. In topological dynamics our notion may be interpreted as a topological version of the classical Rokhlin lemma: automorphisms arising from minimal homeomorphisms of finite dimensional compact metrizable spaces always have finite Rokhlin dimension. The latter result has by now been generalized by Szab\'o to the case of free and aperiodic $\mathbb{Z}^{d}$-actions on compact metrizable and finite dimensional spaces. 
\end{abstract}
\maketitle

\section{Introduction}

\noindent
An action of a countable discrete group $G$ on a (possibly noncommutative) space $X$ has a Rokhlin property if there are systems of subsets of $X$, indexed by large subsets of $G$, 
\begin{enumerate}
\item[(i)] which are (approximately) pairwise disjoint,
\item[(ii)] which (approximately) cover $X$ and
\item[(iii)] on which the group action is (approximately) compatible with translation on the index set.
\end{enumerate}

In a sense, Rokhlin properties create a `reflection' of the acting group in the underlying space which often allows for a dynamic viewpoint on properties of the latter. The concrete interpretation of conditions (i), (ii) and (iii) above  allows for a certain amount of freedom, and may vary with the applications one has in mind. For a general group it is usually not so clear how to synchronize indexing by large subsets and compatibility of the group action with translation on these subsets. The situation is much less ambiguous if $G$ is finitely generated, in particular for $\mathbb{Z}$ (or $\mathbb{Z}^{d}$) or for finite groups.

In this paper we will be mostly interested in the noncommutative case, more specifically in actions on $C^{*}$-algebras. The study of such actions, and their associated
crossed products, has always been a central  theme in
operator algebras, beginning with Murray and von Neumann's group measure space construction which associates a von Neumann algebra to an action of a group on a measure space.
Since then the crossed product construction has been generalized to actions on noncommutative von Neumann and $C^*$-algebras, providing an inexhaustible source of highly nontrivial examples. It combines in an intricate way dynamical properties of the action and properties of the algebra of coefficients.

A basic, yet  important early result in the dynamics of group actions on measure spaces is the Rokhlin Lemma saying that a measure preserving aperiodic (i.e.\ almost everywhere non-periodic) action of $\Z$ can be approximated by cyclic shifts, in the sense that there exists a finite partition of the space (Rokhlin tower) which is almost cyclically permuted. It can be reformulated as a result about strongly outer automorphisms of commutative von Neumann algebras involving partitions of unity by orthogonal projections (towers of projections) and this formulation has led to the various versions of  $C^*$-algebraic Rokhlin properties. (We recall the version which is now most commonly used in Definition~\ref{rokhlin-original}.)

As it turns out there are many important examples of actions of finite groups and $\Z$ possessing  Rokhlin properties.  See, for instance, \cite{Ks1,Iz0} and references therein for actions of $\Z,$ and \cite{HJ2, Iz, Iz2,OP}, for the finite group case. In fact, Rokhlin properties, especially for the single automorphism case, are quite prevalent, and indeed  as a byproduct we establish in this paper that they are generic for automorphisms of unital $C^*$-algebras which absorb a UHF algebra of infinite type. For related work establishing genericity of Rokhlin type conditions, see \cite{Phi:tracial-generic}.

Most currently used  $C^*$-algebraic Rokhlin properties involve towers consisting of projections. However, unlike in the case of von Neumann algebras, requiring the existence of projections poses severe restrictions on the coefficient algebra. In particular, by lack of projections, automorphisms of the Jiang-Su algebra $\mathcal{Z}$ (the smallest possible $C^{*}$-algebraic analogue of the hyperfinite $II_{1}$ factor $\mathcal{R}$) or automorphisms arising from homeomorphisms of connected spaces will not satisfy most of the current definitions of the Rokhlin property.  Moreover, even if there are sufficiently many projections (e.g.\ in the real rank zero case)  Rokhlin properties do not always hold. In the $C^*$-setting they become  regularity properties of  actions which may be regarded as a strong form of outerness (producing mostly simple $C^*$-algebras).

The main purpose of this paper is to develop  generalizations of  Rokhlin properties motivated by the idea of covering dimension. Roughly speaking, instead of requiring towers consisting of projections we allow for partitions of unity involving several towers of positive elements, where elements from different towers are no longer required to be orthogonal, but the number of different towers is restricted. (A further requirement, which we need for some results and which is automatic in many important cases, is that these towers  commute.) This is analogous to the definition of covering dimension using partitions of unity with functions allowing for controlled overlaps. We like to think of the number of towers as the coloring number, or dimension, of the cover  of the dynamical system, the usual Rokhlin property corresponding to the zero dimensional case. It turns out that this generalization provides a much more flexible concept, which applies to a much wider range of examples than the currently used Rokhlin properties. 
We can even show that in the $\Zh$-stable case finite Rokhlin dimension is topologically generic.  
A very nice and explicit example for which we can show finiteness of our Rokhlin dimension and where the usual Rohklin property clearly fails are irrational rotation automorphisms of $C(\T)$. In the initial version of this paper we showed that that automorphisms corresponding to minimal homeomorphisms of compact spaces of finite covering dimension always have finite Rokhlin dimension. %In particular we find that minimal homeomorphisms of the Cantor set always have Rokhlin dimension at most 1 (i.e.\ involving at most two towers consisting of positive elements).   
This result might be regarded as a topological version of the classical Rokhlin Lemma. After this result appeared on the arXiv, Szab\'o in \cite{Sza:dimnuc} generalized it to the case of free and aperiodic $\mathbb{Z}^{d}$-actions. Szab\'o's argument is based on results by Gutman and by Lindenstrauss; it is technically easier than our original approach, which we do not include in the present version (it is still available in arXiv:1209.1618v1).

Covering dimension for topological spaces was generalized to the context of $C^*$-algebras in \cite{kirchberg-winter} and \cite{winter-zach}, in which the related concepts of decomposition rank and nuclear dimension were introduced. Those are based on placing a uniform bound on the decomposability of c.p.\ approximations of the identity map (cf.\ \cite{HKW}). Simple $C^*$-algebras with finite decomposition rank or finite nuclear dimension have been shown to have significant regularity properties, in particular they are $\mathcal{Z}$-stable (\cite{winter-dr-Z,winter-nd-Z}).
Whilst these classes have good permanence properties, in particular those with finite nuclear dimension, there are still no general results known involving crossed products. Our original motivation was to find bounds on the decomposition rank and/or nuclear dimension of crossed products, which is indeed possible for Rokhlin actions of finite groups. Moreover it turns out that similar bounds can be established for automorphisms satisfying higher dimensional Rokhlin properties, at least 
for the nuclear dimension. 
Combining these results with the finiteness of Rokhlin dimension for minimal systems on compact spaces of finite covering dimension we obtain finiteness of the nuclear dimension of crossed products of such systems. This provides a different proof of the nuclear dimension result in \cite{TomsWinter:minhom}. Moreover, the generalization obtained by Szab\'o (see \cite{Sza:dimnuc}) together with the results of \cite{Lin:crossed-product-AF-embedding} have now led to the classification of crossed products associated to Cantor free minimal $\mathbb{Z}^{d}$-actions with compact space of ergodic measures, see \cite{Win:class-cross} (also see \cite{MatSat:dr-UHF} for the uniquely ergodic case). In fact, the respective result holds for actions on compact metrizable and finite dimensional spaces, provided that traces are separated by projections.

In this paper we focus our attention on unital $C^*$-algebras. The non-unital case, as well as further properties of finite Rokhlin dimension for finite group actions, will be explored in a forthcoming paper of the first named author and N.C.~ Phillips, \cite{hirshberg-phillips}. 

We note that other generalizations of the Rokhlin property have been considered  which are based on the idea of leaving out a remainder which is small in a tracial sense (see \cite{OP-tracial,Sato,Archey,hirshberg-orovitz,Phi:tracialRokhlin}).

We also wish to point out that finite Rokhlin dimension can be interpreted as a dynamic version of Kirchberg's covering number, introduced in \cite{Kir:CentralSequences} (also cf.\ \cite{KirRor:pi}).

Hiroki Matui and Yasuhiko Sato have informed us that they have a proof that \emph{all} strongly outer automorphisms of separable, simple, unital, monotracial and $\mathcal{Z}$-stable $C^{*}$-algebras have Rokhlin dimension at most 1, thus partially generalizing our Theorem~\ref{Z-stable-generic}. Their argument is based on the techniques developed in \cite{MatSat:Z}.  

As a further variation of our Rokhlin dimension one might in addition ask the Rokhlin systems to approximate the underlying algebra; one might also compare the necessary number of colors with the possible lengths of the Rokhlin towers  of such approximations (instead of asking for a global bound on the number of colors). This idea leads to a version of `slow dimension growth' which in turn is closely related to the notion of mean dimension introduced by Lindenstrauss and Weiss in \cite{LinWei:meandimension}.

Finally, we mention that Bartels, L\"uck and Reich have studied a concept closely related to ours in \cite{BLR:equivariant}; this idea plays a key role in their proof of the Farrell-Jones conjecture for hyperbolic groups in \cite{BLR:Farrell-Jones}. 

These connections will be studied in subsequent work.

The organization of the paper is as follows.  We first consider the case of finite group actions in Section~\ref{section:finite-groups}, which in some respects is simpler than the case of actions of $\Z$ and serves as a  good motivation displaying some essential ideas. In Section~\ref{section:single-auto-basic-properties} we introduce various versions of Rokhlin dimension for $\mathbb{Z}$-actions and study their interplay.  In Section~\ref{section:genericity} we show that finite Rohklin dimension is generic in the $\Zh$-stable case. A similar argument shows that the Rokhlin property is generic in the UHF-stable case. In Section~\ref{section:permanence-nd} we show that finite nuclear dimension is preserved under forming crossed products by automorphisms with finite Rokhlin dimension, and in Section~\ref{section:permanence-Z} we show that $\Zh$-stability is preserved if we assume a further commutativity condition in the definition of Rokhlin dimension (which holds in all cases we are aware of); the latter result holds for actions of finite groups and of $\mathbb{Z}$.

In Section~\ref{section-irrational} we consider minimal $\mathbb{Z}$-actions on finite dimensional compact Hausdorff spaces, which always turn out to have finite Rokhlin dimension. We give a direct argument of this fact for irrational rotations and compare it to Szab\'o's generalization to free and aperiodic $\mathbb{Z}^{d}$-actions.

%
%The remainder of the paper is devoted to showing that  minimal $\mathbb{Z}$-actions on finite dimensional compact Hausdorff spaces always have finite Rokhlin dimension. 
%As the proof is quite technical, we give a (different) short and direct argument for irrational rotations of the circle in Section~\ref{section-irrational}. The proof of the general result is carried out in Sections~\ref{section-cyclic-non-cyclic} through \ref{section-main-minimal}. The argument closely follows that of \cite{Winter:subhomdr}, computing the decomposition rank for orbit breaking large subalgebras of transformation group $C^{*}$-algebras while at the same time carefully keeping track of the underlying dynamics.  

In an appendix we recall the crossed product construction as well as the notions of nuclear dimension and of $\mathcal{Z}$-stability for the reader's convenience.

We would like to thank the referee for carefully reading the manuscript and for a number of helpful comments.

\tableofcontents

\newpage

\section{Rokhlin dimension for actions of finite groups}
\label{section:finite-groups}

\noindent
In this section we define the concept of Rokhlin dimension for finite group actions on $C^{*}$-algebras and show how this notion can be used to prove  permanence of finite nuclear dimension under taking crossed products.

%We start with the finite group case, since here there is less freedom of choice as in  the case of integer actions, leaving us with a rather unambiguous definition. 

\begin{Def}
\label{def: positive Rokhlin finite groups}
Let $G$ be a finite group, $\A$ a unital $C^*$-algebra and 
\[
\alpha : G \to \textup{Aut} (\A)
\]
an action of $G$ on $\A$.
We say that $\alpha$ has Rokhlin dimension $d$ if $d$ is the least integer such that the following holds: for any 
$\eps >0$ and every finite subset $\F \subset \A$ there are positive contractions 
\[
\left( f^{(l)}_g \right)_{l=0,\ldots,d;\; g \in G} \subset \A
\]
satisfying
\begin{enumerate}
\item for any $l$, $\left\|f_{g}^{(l)}f_{h}^{(l)}\right\|< \eps$, whenever $g \neq h$ in $G$,
\item $\left\|  \sum_{l=0}^{d}\sum_{g \in G}  f_{g}^{(l)} - 1 \right\|<\eps$,
\item $\left\|\alpha_h \Big(f_{g}^{(l)}\Big) - f_{hg}^{(l)} \right\| <\eps$ for all $l \in \{ 0 ,\ldots, d\}$ and 
$g \in G$,
\item $\left\|\big[f_{g}^{(l)},a \big]\right\| < \eps$ for all $l \in \{ 0 ,\ldots, d\}$, $g\in G$ and $a \in \F$.
\end{enumerate} 
\end{Def}
We will refer to the family $( f^{(l)}_g )_{l=0,\ldots,d; \; g \in G}$ as a multiple tower  and to 
the $d+1$ families $( f^{(0)}_g )_{g \in G}, \ldots , ( f^{(d)}_g )_{g \in G}$ as towers of color $0,\ldots,d$, respectively. 
Notice that if $d=0$ then we obtain the usual Rokhlin property for actions of finite groups (see \cite{Iz,Iz2}).

A possible variant of the above definition would be to weaken (2) to 
\begin{enumerate}
\item[(2')] $f:= \sum_{l=0}^{d}\sum_{g \in G}  f_{g}^{(l)} \geq (1- \eps) \cdot 1_{\mathcal{A}}$.
\end{enumerate}
If the towers can be chosen to commute or approximately commute it's not hard to see that if (2') is true then, upon replacing $ f_{g}^{(l)}$ by 
$f^{-1/2} f_{g}^{(l)}f^{-1/2}$ and $\eps$ by a sufficiently small fraction of $\eps$, we can also obtain (2), so that (2) and (2') are equivalent in this case.

\begin{Rmk} \label{square root and compact}
By approximating the function $t \mapsto t^{1/2}$ by polynomials on $[0,1]$ we may  assume in  Definition 
\ref{def: positive Rokhlin finite groups} above that for $h,g \in G$ and $a \in \F$ we have 
$$\left\| {f_{g}^{(l)}}^{1/2} {f_{h}^{(l)}}^{1/2}\right\| < \eps \textup{ if $g \neq h $, } 
\left\|\Big[{f_{g}^{(l)}}^{1/2},a \Big]\right\| < \eps \textup{ and } \left\|\alpha_h \Big({f^{(l)}_g}^{1/2}\Big)-{f^{(l)}_{hg}}^{1/2}\right\| < \eps 
$$
for all $l \in \{ 0, \ldots ,d\}$.
Also, we may replace the finite set $\F$ in the definition by a norm compact set. 
\end{Rmk}

\begin{Thm}
\label{Thm: positive Rokhlin finite groups}
Let $G$ be a finite group, $\A$ a unital $C^*$-algebra with finite decomposition rank and $\alpha:G\to \textup{Aut}(\A)$
be an action  with Rokhlin dimension $d$. 

Then the crossed product $\A \rtimes_{\alpha} G$ has finite 
decomposition rank as well - in fact, $$\dr(\A \rtimes_{\alpha} G) \leq (\dr(\A) +1) (d+1)-1.$$ 
The respective statement is true 
for  nuclear dimension.
\end{Thm}

\begin{proof}
We will only give a proof for the decomposition rank; the proof for  nuclear dimension is similar. 
Set $n = |G|$ and recall that $\dr(M_n(\A)) = \dr(\A)$ and denote this number by $N$.

Let $\F \subset \A \rtimes_{\alpha} G \subset M_n(\A)$ finite and $\eps>0$ be given; we may assume that $\F$ consists of contractions. Choose an $N$-decomposable 
c.p.c.\ approximation (for $M_{n}(\mathcal{A})$) consisting of  $\mathcal{F}  =\mathcal{F}  ^{(0)} \oplus \ldots \oplus \mathcal{F}  ^{(N)}$,
$\psi : M_n(\A) \to \mathcal{F}  $, $\varphi:\mathcal{F} \to M_n(\A)$,  
$\varphi = \varphi^{(0)} +  \ldots + \varphi^{(N)}$ for $\F$ to within $\eps/6$. 
 
We will construct c.p.\ approximations (for $ \A \rtimes_{\alpha} G$) as below, and then use stability of order zero maps to replace the maps $\rho \varphi $ by ones that are decomposable into order zero maps:
$$
\xymatrix{
\A\rtimes_{\alpha} G \ar[dr]^{\iota} &   & &  &  \A\rtimes_{\alpha} G \\ 
  &  M_n(\A)    \ar[dr]^{\bigoplus_{j=0}^d \psi} &  &  \bigoplus_{j=0}^{d} M_n(\A)   \ar[ur]^{\rho} &   \\
  &  & \bigoplus_{j=0}^{d} \mathcal{F}  \ar[ur]^{\bigoplus_{j=0}^d \varphi}  
}
$$

Let $K_0 = \bigcup_{g \in G} \bigcup_{j=0}^N (\textup{id}_{M_n} \otimes\alpha_g) (\varphi^{(j)}(B_\mathcal{F}  ))$, 
where $B_\mathcal{F}  $ is the closed unit ball in $\mathcal{F}  $, and let $K \subseteq \A$ be the collection of all matrix 
entries of elements in $K_0$. Notice that $K_0,K$ are compact subsets of $M_n(\A)$
and $\A$ respectively. Using stability of order zero maps, we choose $\eta >0$ such that $\eta < \eps/6$ 
and such that if $\theta:\mathcal{F}   \to \A \rtimes_{\alpha} G$ is a c.p.c.\:map which satisfies 
$\|\theta(a)\theta(b)\|\leq (n+4)n^3\eta$ for all positive orthogonal contractions $a,b \in \mathcal{F}  $, 
then there exists a c.p.c.\:order zero map $\theta':\mathcal{F}   \to  \A \rtimes_{\alpha} G$ such that 
$\|\theta' - \theta\|<\eps/(6(d+1)(N+1))$. Finally we also require $\eta$ to be such that 
$(2(d+1)n +1)n\eta < \eps /6$ which will be used later. 

Every element $x \in \F \subseteq \A \rtimes_{\alpha} G $ can be written  in the form 
$x= \sum_{g \in G} x(g) u_g$, where 
$x(g) \in \A$ are uniquely determined contractions. Let $\tilde{\F}$ be the finite set  
$\tilde{\F}= \{ x(g) \mid x \in \F, g \in G \} \subseteq \A$.

For the particular $\eta >0$ we have chosen and the compact set $K \cup \tilde{\F}$, choose a 
multiple tower $\left( f^{(l)}_g \right)_{l=0,\ldots,d; \: g \in G}$ such that (1)--(4) 
of Definition \ref{def: positive Rokhlin finite groups} hold true for $\eps$ replaced by 
$\eta$ and $a \in K \cup \tilde{\F}$ and also 
$\Big\| {f_{g}^{(l)}}^{1/2} {f_{h}^{(l)}}^{1/2}\Big\| < \eta$ ($g,h \in G$ different), 
$\Big\|\alpha_h({f^{(l)}_g}^{1/2})-{f^{(l)}_{hg}}^{1/2}\Big\|<\eta$ and 
$\Big\|[{f_{g}^{(l)}}^{1/2},a]\Big\| < \eta$ for all $l \in \{ 0, \ldots ,d\}$, 
$a \in K \cup \tilde{\F}$, $g , h \in G$.

Now, define $\rho^{(l)}:M_n(\A) \to \A \rtimes_{\alpha} G$ by
$$
\rho^{(l)}(e_{g,h} \otimes a) = {f^{(l)}_g}^{1/2} u_g a u_{h}^*{f^{(l)}_h}^{1/2}.
$$
We note that $\rho^{(l)}$ is c.p.\:since it is the composition of the canonical inclusion 
$M_n \otimes \A \hookrightarrow M_n \otimes \A \rtimes_{\alpha} G$ with the map 
$M_n( \A \rtimes_{\alpha} G) \to \A  \rtimes_{\alpha}  G$ given by $a \mapsto v_{l} av_{l}^*$, 
where $v_{l} \in M_{1,n}( \A \rtimes_{\alpha} G )$ is defined by $v_{l} = 
\left( {f^{(l)}_g}^{1/2} \;  u_g  \; \right)_{g \in G}$ (as a row vector whose 
entries are indexed by the elements of $G$). Thus the sum $\rho = \sum_{l =0}^d \rho^{(l)}$ 
is also c.p. For $a \in K \cup \tilde{\F}$, $g \in G$, $l \in \{0, \ldots ,d\}$ we have 
\begin{eqnarray*}
\rho^{(l)} (au_g) 
& = &
\rho^{(l)}\Big(\sum_{h \in G} e_{h,g^{-1}h} \otimes \alpha_{h^{-1}}(a)\Big) \\
& = & 
\sum_{h \in G} {f_h^{(l)}}^{1/2} u_h \alpha_{h^{-1}} (a) u_{h^{-1}g} {f_{g^{-1}h}^{(l)}}^{1/2} \\
& = & 
\sum_{h \in G} {f_h^{(l)}}^{1/2} a u_g {f_{g^{-1}h}^{(l)}}^{1/2} \\
& =_{n\eta} &
\sum_{h \in G} {f_h^{(l)}}^{1/2} a {f_{h}^{(l)}}^{1/2} u_g  \\
& =_{n\eta} &
\sum_{h \in G} {f_h^{(l)}} a  u_g. 
\end{eqnarray*}
Thus $\| \rho (a u_g) - au_g \| < (2(d+1)n+1)\eta$ by assumption (note that $\|a\| \leq 1$) and so 
$$
\| \rho (x) - x \| < (2(d+1)n+1)n \eta <  \eps/6 
$$ 
for $x \in \F$.

Note that $\rho(1)=\sum_{l =0}^d \rho^{(l)} (1) = \sum_{l=0}^{d}\sum_{g \in G}  f_{g}^{(l)}$, 
therefore  $\| \rho \| < 1+ \eps/6$ and so after possible normalization (dividing by its norm)  we have 
$\| \rho \| = 1$ and $\| \rho (x) - x \| < \eps/3$ for all $x \in \F$. 

Define  $f^{(l)}:= \sum_{g \in G} f^{(l)}_g$, where $l \in \{0, \ldots ,d\}$. One checks that for all $a,b \in K \cup \tilde{\F}$, 
and $g_1,g_2,h_1,h_2 \in G$, we have 
$$
\| \rho^{(l)}(a \otimes e_{g_1,h_1})\rho^{(l)}(b \otimes e_{g_2,h_2} ) - f^{(l)} \rho^{(l)} ((a \otimes e_{g_1,h_1}) \cdot (b \otimes e_{g_2,h_2}) )\| < (n+4)\eta . 
$$
Therefore, if $a,b$ are elements of $K_0 \subset M_n(\A)$, then 
$$
\|\rho^{(l)}(a)\rho^{(l)}(b) - f^{(l)} \rho^{(l)}(ab) \| \leq (n+4)n^3 \eta.
$$
Consequently, for any $j=0,\ldots,d$, we have that 
$$
\big\|\rho^{(l)}(\varphi^{(j)}(a)) \cdot \rho^{(l)}( \varphi^{(j)}(b)) - f^{(l)} \rho^{(l)}(\varphi^{(j)}(a)\varphi^{(j)}(b))\big\| < (n+4)n^3 \eta
$$ 
for all $a,b \in B_\mathcal{F}  $. In particular, if $a,b \in B_\mathcal{F}  $ are positive and orthogonal, then $\|\rho^{(l)}(\varphi^{(j)}(a))\rho^{(l)}(\varphi^{(j)}(b))\| \leq (n+4)n^3 \eta$, and $\rho^{(l)} \circ \varphi^{(j)}$ 
is a c.p.c.\  map. By the choice of $\eta$, there are c.p.c.\ order zero maps 
\[
\sigma^{(l,j)}:\mathcal{F}   \to   \A  \rtimes_{\alpha}  G
\]
such that 
\[
\|\sigma^{(l,j)} - \rho^{(l)} \circ \varphi^{(j)}\|<\eps/[6(d+1)(N+1)].
\] 
Set 
\[
\sigma^{(l)}:= \sum_{j=0}^{N}\sigma^{(l,j)} : \mathcal{F} \to  \A  \rtimes_{\alpha}  G
\]
and
$$\sigma = \sum_{l=0}^{d} \sigma^{(l)}: \bigoplus_{l=0}^{d} \mathcal{F} \to  \A  \rtimes_{\alpha}  G.$$ If $\|\sigma\|>1$, then we normalize $\sigma$; since $\|\sigma\|<1+\eps/6$, normalizing will again introduce an error of at most 
$\eps/6$.

Thus, possibly after normalizing, we have that $\|\sigma - \rho \circ \varphi \|< \eps/3$, and putting everything together, the triple $(\bigoplus_{l=0}^{d }\mathcal{F}  , \psi , \sigma)$ is a 
$((d+1)(N+1)-1)$-decomposable approximation for $\F$ within $\eps$, as required.
\end{proof}

\begin{Rmk}
We will see in Theorem~\ref{finite-action-Z-stable} that $\mathcal{Z}$-stability also passes to crossed products by finite group actions with finite Rokhlin dimension, provided one in addition assumes that the Rokhlin towers commute.
\end{Rmk}

\section{Rokhlin dimension for actions of $\Z$}
\label{section:single-auto-basic-properties}

\noindent
In this section we introduce the concept of Rokhlin dimension for $\Z$-actions on $C^{*}$-algebras. There is a certain amount of freedom in the definition, just as for the original Rokhlin property. For the latter, several versions have been studied (cf.\ \cite{Ks1}, \cite{Iz0} and the references therein). Correspondingly,  we introduce several versions  of Rokhlin dimension and compare these. As it turns out, in general one such version is finite if and only if the other is -- roughly, they  bound one another by a factor $2$.   Since we are mostly interested in when the Rokhlin dimension is finite (as opposed to the precise value), for our purposes it will not be too important which variant we choose.      The different versions of Rokhlin dimension are designed so that they are  just the zero dimensional instances  of  the various  Rokhlin properties; in this regard,  our concept  bridges the gaps between the latter.   For our exposition we will give slight preference to the version  generalizing  the definition below.

Let us first recall  the Rokhlin property which is now most commonly used.

\begin{Def}
\label{rokhlin-original}
Let $\A$ be a unital $C^*$-algebra, and let $\alpha:\Z \to \mathrm{Aut}(\mathcal{A})$ be an action of the integers. We say that $\alpha$ has the Rokhlin property if for any 
finite set $\F \subseteq \A$, any $\eps>0$ and any $p \in \N$ there are projections 
\[
e_{0,0},\ldots,e_{0,p-1},e_{1,0},\ldots,e_{1,p} \in \A
\]
such that 
\begin{enumerate}
\item $\sum_{r=0}^1 \sum_{j=0}^{p+r-1} e_{r,j} = 1$,
\item $\|\alpha_{1}(e_{r,j})-e_{r,j+1}\|<\eps$ for all $j=0,1,\ldots,p-2+r$,
\item $\|[e_{r,j},a]\|<\eps$ for all $r,j$ and all $a \in \F$.
\end{enumerate}
We shall refer to the projections in the definition as Rokhlin projections and to the set of those projections as 
Rokhlin towers. 
\end{Def}

\begin{Rmks}
\label{Rokhlin-rems}
(i) It follows immediately from the definition that the Rokhlin projections are pairwise orthogonal, 
and that $\|\alpha_{1}(e_{0,p-1} + e_{1,p}) - (e_{0,0} + e_{1,0})\|<\eps$. In the generalization we consider, we will need to add such 
assumptions explicitly. 

(ii) It could be that the $e_{1,j}$'s (or the $e_{0,j}$'s) are all zero. 
If this can always be arranged then we say that $\alpha$ has the single tower  Rokhlin property. The single tower Rokhlin property 
is stronger than the Rokhlin property - it implies immediately that the identity can be decomposed into a sum of projections which are 
equivalent via an automorphism, which amounts to a nontrivial divisibility requirement. For instance, there can be no automorphism on the 
UHF algebra of type $2^{\infty}$ which has the single tower Rokhlin property as stated here, whereas we shall 
see that the Rokhlin property is generic for automorphisms of this algebra (for this $C^*$-algebra, there are many automorphisms with 
single Rokhlin towers if one restricts to towers of a height which is a power of 2 -- in fact, basically the same proof shows that such automorphisms are generic). 
\end{Rmks}

We now turn to our main definition.

\begin{Def} \label{Rokhlin-dim-single-auto}
Let $\A$ be a unital $C^*$-algebra and  $\alpha:\Z \to \mathrm{Aut}(\mathcal{A})$ an action of the integers. 

a) We say $\alpha$ has Rokhlin dimension   
 $d$, 
 \[
 \dimrokh(\A, \alpha)=d,
 \]
if $d$ is the least natural number  such that the following holds: 
 
 For any finite subset $\F \subset \A$, 
any $p \in \N$ and any $\eps>0$, there are positive elements 
\[
f_{0,0}^{(l)},\ldots,f_{0,p-1}^{(l)}, f_{1,0}^{(l)},\ldots,f_{1,p}^{(l)}, \; l \in \{ 0,\ldots,d\},
\]
satisfying
\begin{enumerate}
\item for any $l \in \{ 0,\ldots,d\}$, $\big\|f_{q,i}^{(l)}f_{r,j}^{(l)}\big\|<\eps$, whenever $(q,i) \neq (r,j)$,
\item $ \left\| \sum_{l=0}^{d}\sum_{r=0}^1\sum_{j=0}^{p-1+r} f_{r,j}^{(l)} - 1 \right\|<\eps$,
\item $\big\|\alpha_{1}(f_{r,j}^{(l)}) - f_{r,j+1}^{(l)}\big\|<\eps$ for all $r \in \{0,1\}$, $j\in \{0,1,\ldots,p-2+r\}$ and all $l \in \{0,\ldots,d\}$,
\item  $\big\|\alpha_{1}(f_{0,p-1}^{(l)} + f_{1,p}^{(l)}) - (f_{0,0}^{(l)} + f_{1,0}^{(l)})\big\|<\eps$ for all $l$,
\item $\big\|[f_{r,j}^{(l)},a]\big\| < \eps$ for all $r,j,l$ and $a \in \F$.
\end{enumerate} 

If there is no such $d$ then we say that $\alpha$ has infinite Rokhlin dimension and write $\dimrokh (\A,\alpha) = \infty$. 

b) If, moreover, for all towers we can arrange $\|[f_{q,i}^{(l)},f_{r,j}^{(m)}]\|< \eps$ for all $q,i,l,r,j,m$, then 
we  say that $\alpha$ has  Rokhlin dimension $d$ with commuting towers, 
\[
\dimrokhct(\A,\alpha)=d.
\] 

c) If one can obtain this property such that for all $l$ one of the towers $(f_{0,j}^{(l)} )_j$ or $(f_{1,j}^{(l)})_j$ vanishes, then we shall say 
that $\alpha$ has Rokhlin dimension $d$ with single towers, 
\[
\dimrokhsingle(\A,\alpha)=d.
\] 

d) We write 
\[
\dimrokhbar (\mathcal{A},\alpha) = d
\]
if $\alpha$ satisfes the conditions above except that (2) is replaced by the  weaker condition
\begin{itemize}
\item[(2')] $\sum_{l=0}^{d}\sum_{r=0}^1\sum_{j=0}^{p-1+r} f_{r,j}^{(l)} \ge (1-\eps) \cdot 1_{\mathcal{A}}$.
\end{itemize}

e) We define $\dimrokhbarct (\mathcal{A},\alpha)$ and $\dimrokhbarsingle (\mathcal{A},\alpha)$ in the respective manner, replacing property (2) by (2') in b) and c).

We shall refer to each sequence  $(f_{r,j}^{(l)} )_{j\in \{0,\ldots,p-1+r\}}$ 
as a tower, to the length of the sequence as the height of the tower, and to the pair of towers $(f_{r,j}^{(l)} )_{j \in \{0,\ldots,p-1+r\}}$, $r \in \{0,1\}$, as a double tower. 
The superscript $l$ will sometimes be referred to as the color of the tower.
\end{Def}

\begin{Rmks}
\label{Rokhlin-rems2}
(i) Obviously $\dimrokh(\A,\alpha)$ and  $\dimrokhct(\A,\alpha)$ coincide if $\A$ is commutative and we will show below (Proposition~\ref{dominated-Rokhlin}) that in this case we also have 
$\dimrokhbar (\A,\alpha) = \dimrokh (\A,\alpha)$ and $\dimrokhbar^{\mathrm{s}} (\A,\alpha) = \dimrokh^{\mathrm{s}} (\A,\alpha)$.

(ii) Note that in the single tower version defined in c) above, property (4) of \ref{Rokhlin-dim-single-auto}a) forces property (3) to hold cyclically, i.e., for $j \in \{0, 1, \ldots , p-1+r\}$.

(iii) It is clear that in general $\dimrokh(\A,\alpha) \leq \dimrokhct(\A,\alpha)$. However in all cases in which we establish that a $\mathbb{Z}$-action has Rokhlin 
dimension $d$ the towers are commuting, 
and we do not know if there are any examples for which there is a strict inequality. The additional assumption concerning 
commuting towers is needed as a hypothesis for some results 
but not others.

(iv) Similarly, $\dimrokhbar (\A,\alpha) \leq  \dimrokh(\A,\alpha)$ in general and it is straightforward to verify that $\dimrokh(\A,\alpha) = 0$ 
if and only if $\alpha$ has the Rokhlin property (if and only if $\dimrokhct(\A,\alpha) = 0$).

(v) As is the case for the regular Rokhlin property, one could have defined a notion of a Rokhlin dimension using more than two towers for each color. That is, require 
that for any given $N$ one could find numbers $R_N$, $p_0,\ldots,p_{R_N}$ such that $p_r>N$ and positive elements $f_{r,j}^{(l)}$, $l = 0,1,\ldots,d$, $r = 0,1,\ldots,R_N$, 
for all $r$ and $j = 0,1,\ldots,p_r$ such that the analogous definition holds (of course, if one allows for arbitrary tower lengths and numbers of towers, then $\eps$ in conditions (1)--(5) has to be replaced by something like $\eps/(p_{0}+\ldots+p_{R_{N}})$). However, this can be reduced to the case of two towers in a standard way, which we 
quickly outline. For a given $p$, we find an $N$ such that any $n>N$ can be written as $a(p-1) + bp$ for some given non-negative integers $a,b$. Using this $N$, 
one finds $f_{r,j}^{(l)}$ as above. We fix
$a_r,b_r$ such that $a_r(p-1) + b_rp = p_r+1$, and then define $$\tilde{f}_{0,j}^{(l)} = \sum_{r=0}^{R_N} \sum_{m=0}^{a_r-1}f_{r,m(p-1)+j}^{(l)} $$
$$\tilde{f}_{1,j}^{(l)} = \sum_{r=0}^{R_N} \sum_{m=0}^{b_r-1}f_{r,a_r(p-1) + mp+j}^{(l)}$$
and this reduces the case of an arbitrary number of towers of each color, each in an arbitrary length of at least $N$ to our notion of Rokhlin dimension.

(vi) The definition of Rokhlin dimension (and Rokhlin dimension with commuting towers) is equivalent to the following apriori weaker formulation: instead of asking for double towers of arbitrary height $p$, one can  require that there are such towers for arbitrarily large $p$ (independently of $\eps$) but not necessarily for any $p$. This can be shown by using a rearrangement of tower terms as explained above, and is useful in certain examples  (e.g.\ if it is easier to produce towers which are of prime height, or of height which is a power of 2). We note, however, that in the single tower case this indeed gives a different formulation. For example (cf.\ Remark~\ref{Rokhlin-rems}(ii)), the infinite tensor shift on the CAR algebra has the Rokhlin property and one can obtain single towers of height $2^n$, but if one requires heights which are not powers of $2$ then one needs two towers (this follows from $K$-theoretic considerations).  
\end{Rmks}

\begin{Not}
We like to think of finite Rokhlin dimension as a property of a group action rather than of an automorphism, even if the group is cyclic. However, when it is understood that the group is $\mathbb{Z}$ we will sometimes slightly abuse notation and  write  $\dimrokh(\A,\alpha)$ even when $\alpha \in \mathrm{Aut}(\A)$  denotes the automorphism inducing the $\mathbb{Z}$-action rather than the action itself. In this situation we also say that the automorphism has finite Rokhlin dimension.
\end{Not}

There is a useful and straightforward reformulation of the definition of Rokhlin dimension which will be used later on:

\begin{Prop}
\label{dimrok-def}
Let $\mathcal{A}$ be a unital $C^{*}$-algebra and $\alpha: \Z \to  \mathrm{Aut} (\mathcal{A})$ a $\Z$-action. 

{\rm a)} Then $\dimrokh (\mathcal{A},\alpha) \le d$ if and only if the following holds:

For every $p \in \mathbb{N}$,  $\F \subset \mathcal{A}$ finite and $\eps >0$  there are c.p.c.\ maps
\[
\zeta^{(l)}: \mathbb{C}^{p} \oplus \mathbb{C}^{p+1} \to \mathcal{A}, \; l \in \{0,\ldots,d\},
\]
such that
\begin{enumerate}
\item $\| \zeta^{(l)}(e) \zeta^{(l)}(e')\| \le \eps \|e\| \|e'\|$ whenever $e,e' \in \mathbb{C}^{p} \oplus \mathbb{C}^{p+1}$ are orthogonal positive elements,
\item $ \| \sum_{l} \zeta^{(l)}(1_{\mathbb{C}^{p}} \oplus 1_{\mathbb{C}^{p+1}}) - 1_{\mathcal{A}} \| \le \eps $,
\item $\|\alpha_{1}(\zeta^{(l)} (e)) - \zeta^{(l)}(\tilde{\sigma}(e)) \| \le \eps \|e\|$ for all elements  $0 \le e \in  \mathbb{C}^{p} \oplus \mathbb{C}^{p+1} $ satisfying $e \perp (e_{p},f_{p+1})$, 
where $\tilde{\sigma} \in \mathrm{Aut} (\mathbb{C}^{p} \oplus \mathbb{C}^{p+1} )$ is the cyclic shift on each summand and $e_{j},f_{j}$ denote the canonical generators of $\mathbb{C}^{p}$ and $\mathbb{C}^{p+1}$, respectively,
\item  $\|\alpha_{1}(\zeta^{(l)} ((e_{p},f_{p+1}))) - \zeta^{(l)}((e_{1},f_{1})) \| \le \eps$,
\item $\| [\zeta^{(l)}(e),a ]\| \le \eps \|e\|$ for all $0 \le e \in \mathbb{C}^{p} \oplus \mathbb{C}^{p+1} $ and $a \in \F$.
\end{enumerate}

{\rm b)} We have $\dimrokhsingle (\mathcal{A},\alpha) \le d$ if for every $p$, $\F$ and $\eps$ as above there are $\zeta^{(0)},\ldots,\zeta^{(d)}$ as above such that each  $\zeta^{(l)}$ vanishes on one of the summands $\mathbb{C}^{p}$ or $\mathbb{C}^{p+1}$.

{\rm c)} We have $\dimrokhct (\mathcal{A},\alpha) \le d$ if for every $p$, $\F$ and $\eps$ as above there are $\zeta^{(0)},\ldots,\zeta^{(d)}$ as above with $\| [\zeta^{(l)} (e),\zeta^{(l')}(e')]\|   \le \eps \|e\| \|e'\|  $   for all $l,l' \in \{0,\ldots,d\}$ and $e,e' \in \mathbb{C}^{p} \oplus \mathbb{C}^{p+1}$.

{\rm d)} We have $\dimrokhbar (\mathcal{A},\alpha) \le d$ if for every $p$, $\F$ and $\eps$ as  in {\rm a)} there are $\zeta^{(0)},\ldots,\zeta^{(d)}$ as  in {\rm a)} satisfying
\begin{itemize}
\item[(2')] $\sum_{l} \zeta^{(l)}(1_{\mathbb{C}^{p}} \oplus 1_{\mathbb{C}^{p+1}}) \ge (1-\eps) \cdot 1_{\mathcal{A}}$
\end{itemize}
instead of {\rm (2)}. 

A combination of {\rm b)}, {\rm c)} and {\rm d)} yields the respective statements  for $\dimrokhbarsingle (\mathcal{A},\alpha)$ and $\dimrokhbarct (\mathcal{A},\alpha)$.
\end{Prop}

When the underlying $C^{*}$-algebra is commutative, then clearly each version of Rokhlin dimension agrees with its commuting tower counterpart -- but in this case, we even have that $\dimrokh$ agrees with $\dimrokhbar$:

\begin{Prop}
\label{dominated-Rokhlin}
Let $(T,h)$ be a dynamical system with $T$ compact and metrizable; let $\alpha$ be the induced action on $C(T)$. Then,
\begin{equation}
\label{dominated-Rokhlin-1}
\dimrokhbar (C(T),\alpha) = \dimrokh(C(T),\alpha)
\end{equation}
and
\[
\dimrokhbarsingle (C(T),\alpha) = \dimrokhsingle (C(T),\alpha).
\]
\end{Prop}

\begin{proof}
We clearly have $\dimrokhbar (C(T),\alpha) \le \dimrokh(C(T),\alpha)$. For the reverse inequality, assume 
\[
\dimrokhbar (C(T),\alpha) \le d
\]
and suppose we are given $ p \in \mathbb{N}$, $ 0 < \bar{\eps} \le 1$ and $\F \subset C(T)$ finite  (in fact, the subset $\F$ is irrelevant here since $C(T)$ is commutative).

Define $f \in C([0,d+1])$ by 
\[
f(t):= 
\left\{
\begin{array}{ll}
t^{-1} & \mbox{if } t \ge 1/2 \\
0 & \mbox{if } t=0\\
\mbox{linear} & \mbox{else}.
\end{array}
\right.
\]
By approximating $f$ uniformly by polynomials we can find 
\[
0< \delta \le \frac{\bar{\eps}}{4}
\]
such that, for any positive $a$ with $\|a\|<d+1$, if $\| \alpha_{1}(a)-a\| < \delta$ then $\|\alpha_{1}(f(a)) - f(a)\|<\frac{\bar{\eps}}{2}$. 

%Whenever $a,b$ are elements in a $C^{*}$-algebra, with $a$ positive, $\|a\| \le d+1$ and $\|b \|\le 1$, and satisfying $\|[a,b ]\| \le \delta$, then $\|[f(a),b ]\| \le \bar{\eps}/2$. (This was first observed in \cite{Arveson:ext}.) 
Set 
\[
\eps:= \frac{\delta}{2(d+1)};
\]
since $\dimrokhbar (C(T),\alpha) \le d$, there are c.p.c.\ maps
\[
\zeta^{(l)}: \mathbb{C}^{p} \oplus \mathbb{C}^{p+1} \to C(T) (=\mathcal{A})
\]
satisfying \ref{dimrok-def}(1), (2'), (3), (4) and (5).  
Note that by \ref{dimrok-def}(2') and since the $\zeta^{(l)}$ are c.p.c.\ we have 
\[
\frac{1}{2} \cdot 1_{C(T)} \le \sum_{l=0}^{d} \zeta^{(l)}(1) \le (d+1) \cdot 1_{C(T)},
\]
whence
\begin{equation}
\label{ww1}
f\left(\sum_{l=0}^{d} \zeta^{(l)}(1) \right) = \left(\sum_{l=0}^{d} \zeta^{(l)}(1) \right)^{-1}.
\end{equation}
From \ref{dimrok-def}(3), (4) we see that 
\[
\left\| \alpha_{1}\left(\sum_{l=0}^{d} \zeta^{(l)}(1) \right) - \sum_{l=0}^{d} \zeta^{(l)}(1) \right\| \le 2(d+1) \eps = \delta.
\]
%This in turn yields the estimate
%\[
%\left\| u \sum_{l=0}^{d} \zeta^{(l)}(1) - \sum_{l=0}^{d} \zeta^{(l)}(1) u \right\|  \le \delta
%\]
%in the crossed product $C(T) \rtimes_{\alpha} \mathbb{Z} (= C^{*}(C(T),u))$, where $u$ denotes the unitary implementing the action $\alpha$. 
By the choice of $\delta$ and by \eqref{ww1} this entails 
\begin{equation}
\label{ww17}
\left\|  \alpha_{1} \left( \left(\sum_{l=0}^{d} \zeta^{(l)}(1)\right)^{-1} \right)- \left(\sum_{l=0}^{d} \zeta^{(l)}(1) \right)^{-1}     
\right\| \le \frac{\bar{\eps}}{2}.
\end{equation}
Define maps 
\[
\bar{\zeta}^{(l)} :\mathbb{C}^{p} \oplus \mathbb{C}^{p+1} \to C(T)
\]
by 
\[
\bar{\zeta}^{(l)}(\, .\,) := \left( \sum_{k=0}^{d} \zeta^{(k)}(1)    \right)^{-1} \zeta^{(l)}(\, .\,)
\]
for $l=0,\ldots,d$; it is clear that the $\bar{\zeta}^{(l)}$ are c.p.c.\ and that 
\[
\sum_{l=0}^{d} \bar{\zeta}^{(l)}(1) = 1_{C(T)},
\]
whence \ref{dimrok-def}(2) holds for the $\bar{\zeta}^{(l)}$ (in fact, with $0$ in place of $\eps$).  

By \ref{dimrok-def}(1), (2') and since $C(T)$ is commutative, we have 
\[
\| \bar{\zeta}^{(l)}(e) \bar{\zeta}^{(l)}(e') \| \le 4 \eps \|e\| \|e'\| \le \frac{\bar{\eps}}{d+1} \|e\| \|e'\|
\] 
whenever $e,e' \in \mathbb{C}^{p} \oplus \mathbb{C}^{p+1}$ are positive orthogonal elements; it follows that \ref{dimrok-def}(1) holds for the 
$\bar{\zeta}^{(l)}$ and $\bar{\eps}$ in place of $\zeta^{(l)}$ and $\eps$. 

\ref{dimrok-def}(5)  holds automatically for the $\bar{\zeta}^{(l)}$ (again with $0$ in place of $\eps$) since $\mathcal{F}  $ and the range of 
the $\bar{\zeta}^{(l)}$ are in $C(T)$. 

Finally, we estimate for each $l \in \{0,\ldots,d\}$ and $0 \le e \in \mathbb{C}^{p} \oplus \mathbb{C}^{p+1}$
\begin{eqnarray*}
\lefteqn{ \left\| \alpha_{1} ( \bar{\zeta}^{(l)}(e)   )  - \bar{\zeta}^{(l)}(\tilde{\sigma}(e))   \right\| } \\
%& = & \left\| u \bar{\zeta}^{(l)} (e) u^{*} - \bar{\zeta}^{(l)}(\tilde{\sigma}(e)) \right\| \\
& = & \left\| \alpha_{1} \left(\left( \sum_{k=0}^{d}  \zeta^{(k)}(1)  \right)^{-1} \zeta^{(l)}(e)  \right)- \left( \sum_{k=0}^{d} \zeta^{(k)}(1)  
\right)^{-1} \zeta^{(l)}(\tilde{\sigma}(e)) \right\| \\
& \stackrel{\eqref{ww17}}{\le} &  \left\|  \left( \sum_{k=0}^{d} \zeta^{(k)}(1)   \right)^{-1}   \left( \alpha_{1} (\zeta^{(l)}(e))  - 
\zeta^{(l)} (\tilde{\sigma}(e))   \right)     \right\|  + \frac{\bar{\eps}}{2} \|e\| \\
& \stackrel{\ref{dimrok-def}(3)}{\le} & \left( 2 \eps + \frac{\bar{\eps}}{2} \right) \|e\| \\
& \le & \bar{\eps} \|e\|,
\end{eqnarray*}
so \ref{dimrok-def}(3) holds for $\bar{\zeta}^{(l)}$ and $\bar{\eps}$ in place of $\zeta^{(l)}$ and $\eps$.

We have now shown \eqref{dominated-Rokhlin-1}. Literally the same proof also yields the single tower version.
\end{proof}

Rokhlin dimension with single towers in general does not coincide with Rokhlin dimension. That is already true in the case of the regular Rokhlin property: there are automorphisms with the Rokhlin property but in which requiring double-towers is essential. One way to see this is via $K$-theoretic considerations (cf.\ Remarks~\ref{Rokhlin-rems}(ii) and \ref{Rokhlin-rems2}(vi)): for example, if $\alpha$ is an approximately inner automorphism which has a Rokhlin tower then all Rokhlin projections are equivalent in $K_0$, and therefore if $\alpha$ has the Rokhlin property with single towers then $K_0(\A)$ must be divisible -- and it is well known that there are Rokhlin automorphisms on $C^{*}$-algebras which do not have divisible $K_{0}$, e.g.\ the CAR algebra.  However, if we allow for higher Rokhlin dimensions, then we see that the distinction is relatively mild:

\begin{Prop}
\label{prop:single-Rokhlin-tower}
Let $\A$ be a unital $C^*$-algebra and $\alpha: \mathbb{Z} \to \mathrm{Aut}(\A)$ an action. Then 
\begin{equation}
\label{prop:single-Rokhlin-tower-1}
\dimrokh(\alpha) \leq \dimrokhsingle(\alpha) \leq 2\dimrokh(\alpha)+1 
\end{equation}
and
\begin{equation}
\label{prop:single-Rokhlin-tower-2}
\dimrokhbar(\alpha) \leq \dimrokhbarsingle(\alpha) \leq 2\dimrokhbar(\alpha)+1. 
\end{equation}
\end{Prop}

\begin{proof}
As in \ref{dominated-Rokhlin}, we only show \eqref{prop:single-Rokhlin-tower-1}; the same argument will also yield \eqref{prop:single-Rokhlin-tower-2}.

The left hand inequality is trivial. As for the right hand one, we may assume that $\alpha$ has finite Rokhlin dimension. Let $f_{r,j}^{(l)}$, $l=0,1,\ldots,d$, $r=0,1$, $j=0,1,\ldots,p-1+r$ be Rokhlin tower elements for $\alpha$, with respect to a given $\epsilon>0$, height $p$ and finite set $\F \subseteq \A$. Let us first assume that $p$ is large enough so that $\epsilon>\frac{1}{p-1/2}$.  We shall construct single towers $(g_j^{(l)})$ for $l=0,1,\ldots , 2d$, the same finite set $\F$ and with $2\epsilon$ instead of $\epsilon$. 

Define decay factors  
$\mu_r:\{0,1,\ldots,p-1+r\} \to \R$ by 
$$
\mu_r(n) = 1-\frac{|(p-1+r)/2 - n|}{(p-1+r)/2},
$$
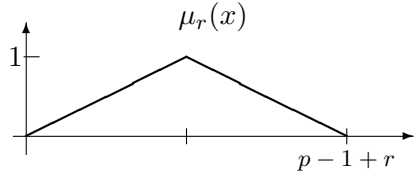
\begin{center}
\begin{picture}(230,65)
 \put(0,10){\vector(1,0){150}}
 \put(5,3){\vector(0,1){50}}
 \put(65,7){\line(0,1){6}}
 \put(125,7){\line(0,1){6}}
 \put(4,40){\line(1,0){6}}
\thicklines
 \put(5,10){\line(2,1){60}}
 \put(65,40){\line(2,-1){60}}
 \put(1,40){\makebox(0,0){$1$}}
 \put(125,-3){\makebox(0,0)[b]{\footnotesize $p-1+r$\normalsize}}
 \put(75,55){\makebox(0,0){$\mu_r(x)$}}
\end{picture}
\end{center}

where we denote $\mu_0(p) = 0$. We use this decay factor to merge each double tower into two separate (cyclic) towers. Schematically, we use the decay factor $\mu_r$ which is large in the middle and small in the ends to bunch up the tower terms as follows
$$
\xy
(0,10)*{f_{0,0}};
(0,0)*{f_{1,0}};
(0,5)*\xycircle(5,10){};
(15,10)*{f_{0,1}};
(15,0)*{f_{1,1}};
(15,5)*\xycircle(5,10){};
(30,10)*{f_{0,2}};
(30,0)*{f_{1,2}};
(30,5)*\xycircle(5,10){};
(45,10)*{\cdots}="";
(45,0)*{\cdots}="";
(60,10)*{f_{0,p-1}}="";
(60,0)*{f_{1,p-1}}="";
(60,5)*\xycircle(7,10){};
(75,0)*{f_{1,p}}="";
(75,00)*\xycircle(5,5){};
\endxy
$$
and we use $1-\mu_r$, which is small at the center, to bunch up the tower terms like this (treat $p$ as even for the purpose of the diagram):
$$
\xy
(0,10)*{f_{0,0}};
(0,0)*{f_{1,0}};
(0,5)*\xycircle(5,10){};
(15,10)*{\cdots};
(15,0)*{\cdots};
(30,10)*{f_{0,\frac{p}{2} -1}};
(30,0)*{f_{1,\frac{p}{2} -1}};
(30,5)*\xycircle(7,10){};
(45,0)*{f_{1,\frac{p}{2}}};
(45,00)*\xycircle(5,5){};
(60,10)*{f_{0,\frac{p}{2}}}="";
(60,0)*{f_{1,\frac{p}{2}+1}}="";
(60,5)*\xycircle(7,10){};
(75,10)*{\cdots}="";
(75,0)*{\cdots}="";
(90,10)*{f_{0,p-1}}="";
(90,0)*{f_{1,p}}="";
(90,5)*\xycircle(6,10){};
\endxy
$$

For $l=0,1,\ldots,d$, and $j=0,1,\ldots,p$, define tower elements as follows, where in the formulas below $f_{r,j}^{(l)}$ for $j>p-1+r$ is meant to be read modulo $p+r$:
$$
g_j^{(l)} = \mu_0(j)f_{0,j}^{(l)} + \mu_1(j)f_{1,j}^{(l)}  \;\;\; , \;\;\; j=0,1,\ldots,p
$$
and for $l=1, \ldots , d$
$$
g_{j}^{(l+d)} = (1-\mu_0(j))f_{0,j}^{(l)} + (1-\mu_1(j))f_{1,j}^{(l)} \;\;\; j=0,1,\ldots,\lceil p/2 \rceil -1
$$
$$
g_{j}^{(l+d)} = (1-\mu_1(j))f_{1,j}^{(l)} \;\;\; \mathrm{for }\;  j=\lceil p/2 \rceil
$$
$$
g_{j}^{(l+d)} = (1-\mu_0(j-1))f_{0,j-1}^{(l)} + (1-\mu_1(j))f_{1,j}^{(l)} \;\;\; j=\lceil p/2 \rceil+1,\ldots ,p.
$$ 
One now readily checks that $g_j^{(l)}$ satisfy the requirements.

%If $p$ is not sufficiently large, we choose a $k$ such that $\epsilon>\frac{1}{kp-3/2}$, and run the procedure above, but also summing up over tower elements with period $k$. We omit the details.
If $p$ is not sufficiently large, we choose a $k$ such that $\epsilon>\frac{1}{kp-3/2}$, choose Rokhlin towers $(f_{r,j}^{(l)})$ for $j=1,...,kp-1+r$ and $l=0,\ldots,d$ with tolerance $\epsilon/k$ instead of $\epsilon$, and repeat the previous procedure to obtain  Rokhlin towers $(g_j^{(l)})$ for $j = 0, 1,..., kp - 1$ and $l=0,\ldots,2d$.
 Defining $\tilde{g}_j^{(l)} = \sum_{n=0}^{k-1}g_{j+np}^{(l)}$ for $j=0,1,\ldots,p-1$ and  $l=0,\ldots,2d$ gives us  Rokhlin towers of height $p$ as required.
\end{proof}

\begin{Rmk}
It follows from the preceding proof  that in fact all the single towers can be chosen to be of the same height (rather than allowing for some to be of height $p$ and others of height $p+1$).
\end{Rmk}

\section{Genericity in the $\Zh$-stable case}
\label{section:genericity}

\noindent
We recall that the automorphism group of a $C^*$-algebra $\A$ can be endowed with the topology of pointwise convergence generated by the  sets 
\[
U_{\alpha,a} = 
\{\beta \mid \|\beta(a)-\alpha(a)\| + \|\beta^{-1}(a) - \alpha^{-1}(a)\|< 1\},
\] 
where $\alpha$ runs over all $\alpha \in \aut(\A)$ and  $a$ runs over all elements of $\A$. If $\A$ is separable, then this topology is Polish.

In this section we want to show that the property of having finite Rokhlin dimension is generic, more precisely the automorphisms  a $\Zh$-stable $C^*$-algebra which have Rokhlin dimension at most 1  form a dense $G_{\delta}$ set in the topology of pointwise norm convergence.

Recall   that for  $p \in \mathbb{N}$  the dimension drop interval $\mathcal{Z}_{p,p+1}$ is given as
\[
\Zh_{p,p+1} = \{f \in C([0,1] , M_p\otimes M_{p+1}) \mid f(0) \in M_p \otimes 1 \; , \; f(1) \in 1 \otimes M_{p+1}\}
\]
and that the Jiang-Su algebra $\mathcal{Z}$ is an inductive limit of such dimension drop intervals.

\begin{Lemma}
\label{Lemma: dimension drop towers}
There exists a unitary $u \in \Zh_{p,p+1}$ and positive elements 
\[
f_0,\ldots,f_{p-1},g_0,\ldots,g_p \in \Zh_{p,p+1}
\]
such that
 \begin{enumerate}
\item $f_i f_j = 0$, $g_i g_j = 0$ and $[f_{i},g_{j}]=0$ for all $i \neq j$.
\item $f_0+\ldots+f_{p-1}+g_0+\ldots+g_p = 1$.
\item $uf_ju^* = f_{j+1}$ and $ug_ju^*=g_{j+1}$ for all $j$, where addition is modulo $p$ or $p+1$, respectively.
\end{enumerate} 
\end{Lemma}

\begin{proof}	
Let $h \in C([0,1])$ be defined by: 
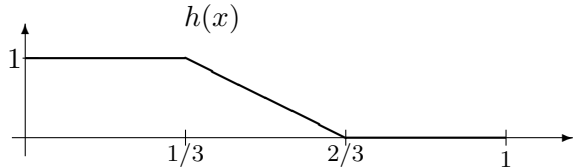
\begin{center}
\begin{picture}(230,65)
 \put(0,10){\vector(1,0){210}}
 \put(5,3){\vector(0,1){50}}
 \put(65,7){\line(0,1){6}}
 \put(125,7){\line(0,1){6}}
 \put(185,7){\line(0,1){6}}
 \put(4,40){\line(1,0){6}}
\thicklines
 \put(5,40){\line(1,0){60}}
 \put(65,40){\line(2,-1){60}}
 \put(125,10){\line(1,0){60}}
 \put(1,40){\makebox(0,0){$1$}}
 \put(65,-1){\makebox(0,0)[b]{\footnotesize $1/3$\normalsize}}
 \put(125,-1){\makebox(0,0)[b]{\footnotesize $2/3$\normalsize}}
 \put(185,-1){\makebox(0,0)[b]{\footnotesize  $1$\normalsize}}
 \put(75,55){\makebox(0,0){$h(x)$}}
\end{picture}
\end{center}
Let $v_x \in M_p$ be a continuous path of unitary matrices such that $v_x$ is a fixed cyclic permutation matrix of order $p$ for all 
$x \in [0,2/3]$, and $v_1 = 1$. Similarly, let $w_x \in M_{p+1}$ be a continuous path of unitaries, which is a fixed cyclic permutation 
matrix of order $p+1$ for all $x \in [1/3,1]$ and $w_0 = 1$. Let $u_x = v_x \otimes w_x$, then $u = (u_x)_{x \in [0,1]}$ is a unitary 
in $\Zh_{p,p+1}$. 
Now, choose projections $a_0,\ldots,a_{p-1} \in M_p$ which add up to 1, such that $v_0 a_j v_0^* = a_{j+1}$ (modulo $p$), and similarly 
choose projections $b_0,\ldots,b_p \in M_{p+1}$ which add up to 1 such that $v_1 b_j v_1^*  = b_{j+1}$ modulo $p+1$. 

Setting $f_j(x) = h(x)\cdot a_x \otimes 1$, $g_j = (1-h(x)) \cdot 1 \otimes b_x$ gives us elements as required.
\end{proof}

\begin{Lemma}
\label{Lemma: stability}
Let $\A$ be a separable $C^{*}$-algebra. Let $a_1,\ldots,a_n,b_1,\ldots,b_n \in \A$, set $c = \sum_{i=1}^n a_i \otimes_{\mathrm{max}} b_i \in \A \otimes_{\mathrm{max}} \A$, and let $p(x,x^*)$ be a noncommutative polynomial. For any $\epsilon>\|p(c,c^*)\|$ there are a $\delta>0$ and a finite subset $W \subseteq \A$ such that the following holds. 

For any $C^*$-algebra $\B$ and any two $^{*}$-homomorphisms $\varphi,\psi:\A \to \B$ which satisfy $\|[\varphi(x),\psi(y)]\|<\delta$ for all $x,y \in W$, if we set $d = \sum_{i=1}^n\varphi(a_i)\psi(b_i)$ then we have that $\|p(d,d^*)\|<\epsilon$.
\end{Lemma}

\begin{proof}
Suppose not. Let $W_m$ be an increasing sequence of finite sets with dense union, then for any $m>0$ we can find a $C^*$-algebra $\B_m$ and a 
pair of homomorphisms $\varphi_m,\psi_m:\A \to \B_m$ such that $\|[\varphi_m(x),\psi_m(y)]\|<1/m$ for all $x,y \in W_m$ but if we set 
$d_m =\sum_{i=1}^n \varphi_m(a_i)\psi_m(b_i)$ then  $\|p ( d_m,d_m^* )\|\geq \eps$. Let $\bar{\varphi} = (\varphi_1,\varphi_2,\ldots)$, 
$\bar{\psi} = (\psi_1,\psi_2,\ldots)$ be the induced homomorphisms $\bar{\varphi},\bar{\psi}:\A \to \prod \B_n / \bigoplus \B_n$. Notice that 
$$
[\bar{\varphi}(x),\bar{\psi}(y)] = 0
$$ for any $x,y$, and therefore we have a well-defined homomorphism $$
\bar{\varphi} \otimes \bar{\psi}: \A \otimes_{\max}\A \to \prod \B_m / \bigoplus \B_m \; .
$$ 
But then it follows that
$$
 \|p( \bar{\varphi}(c), \bar{\varphi}(c)^* )\|\ge \eps
$$
 - contradiction.
\end{proof}

We recall the following simple application of  functional calculus:
\begin{Lemma}
\label{Lemma: unitary stability}
For any $\eps>0$ there is a $\delta>0$ such that for any unital $C^*$-algebra $\A$, if $a \in \A$ satisfies 
$\max\{\|1-a^*a\|,\|1-aa^*\|\}<\delta$ then there is a unitary $u \in \A$ such that $\|u-a\|<\eps$.
\end{Lemma}

\begin{Thm}
\label{Z-stable-generic}
Let $\A$ be a unital separable $\Zh$-stable $C^*$-algebra, then a dense $G_\delta$ set of automorphisms of $\A$ has Rokhlin dimension $\leq 1$ with commuting towers (and in fact with single Rokhlin towers). In particular, Rokhlin dimension at most $1$ is generic.

If furthermore $\A \cong \A \otimes \D$ for $\D$ a UHF algebra of infinite type then the Rokhlin property (i.e.\ Rokhlin dimension $0$) is generic.
\end{Thm}

\begin{proof}
We shall give a proof for the $\Zh$-stable case. The UHF-stable case is similar and will be omitted.
Given $p \in \N$, a finite set $F \subset \A$ and $\eps>0$, we shall say that an automorphism $\alpha \in \mathrm{Aut(\A)}$ has the $(p,F,\eps)$-approximate $1$-dimensional 
Rokhlin property if there are positive elements $f_0,\ldots,f_{p-1},g_{0},\ldots,g_p \in \A$ such that
 \begin{enumerate}
\item $\|f_i f_j\|<\eps$ for all $i \neq j$, and $\|g_i g_j\|<\eps$ for all $i \neq j$,
\item $\|f_0+\ldots+f_{p-1}+g_0+\ldots+g_p -1 \|  <\eps$,
\item $\|\alpha(f_j)-f_{j+1}\|<\eps$ and $\|\alpha(g_j)-g_{j+1}\|<\eps$ for all $j$, where addition is modulo $p$ or $p+1$, respectively,
\item $\|[f_j,a]\|<\eps$ and $\|[g_j,a]\|<\eps$ for all $j$ and all $a \in F$,
\item $\|[f_{i},g_{j}]\| < \eps$ for all $i,j$.
\end{enumerate}  
We denote by $V_{p,F,\eps}$ the set of all automorphisms $\alpha  \in \mathrm{Aut(\A)}$ which satisfy the $(p,\F,\eps)$-approximate $1$-dimensional Rokhlin property.
It is clear that $V_{p,F,\eps}$ is open.

Now, if we choose an increasing sequence of finite sets $F_n \subset \A$ with dense union, then any 
$$
\alpha \in \bigcap_{p \in \N}\bigcap_{n \in \N}V_{p,\F_n,\frac{1}{n}}
$$
has Rokhlin dimension at most 1. It thus suffices to prove that  $V_{p,\F,\eps}$ is dense for any 
$p,\F,\eps$. We thus fix an automorphism $\alpha$ and a triple $(p,\F,\eps)$. For any finite set $\F_0 \subset \A$ and $\gamma>0$ we need to find $\beta \in V_{p,\F,\eps}$ such that 
$$
\max\{\|\beta(a)-\alpha(a)\|,\|\beta^{-1}(a)-\alpha^{-1}(a)\|\} < \gamma
$$ 
for all $a \in \F_0$. Since enlarging $\F$ simply imposes additional conditions, we may assume without loss of generality that $\F \supseteq \F_0$, for notational convenience, and we furthermore assume that all elements of $\F$ have norm at most 1. We may furthermore assume that $\gamma<\eps$.

Fix $\gamma/10 > \eta > 0$ as in Lemma~\ref{Lemma: unitary stability} such that for any unital $C^*$-algebra $\A$, if $a \in \A$ satisfies $\max\{\|1-a^*a\|,\|1-aa^*\|\}<\eta$ then there is a unitary $u \in \A$ such that $\|u-a\|<\gamma/10$.

Choose a unitary $u \in \Zh_{p,p+1} \subset \Zh$ and elements $f_0,\ldots,f_{p-1},g_0,\ldots,g_p \in \Zh_{p,p+1} \subset \Zh$ as in Lemma~\ref{Lemma: dimension drop towers}. Let $G =\{u,f_0,\ldots,f_{p-1},g_0,\ldots,g_p\}$. 

Choose $w \in \Zh \otimes \Zh$ such that $\|w(x \otimes 1)w^* - 1 \otimes x\|<\eta$ for all $x \in G$. Find contractions $a_1,\ldots,a_n,b_1,\ldots,b_n \in \Zh$ such that $\|w - \sum_{i=1}^na_i \otimes b_i\|<\eta$. Notice that 
$$
\left\|1-\sum_{i,j=1}^{n}a_ia_j^* \otimes b_ib_j^*\right\|, \, \left\|1-\sum_{i,j=1}^{n}a_i^*a_j \otimes b_i^*b_j \right\|<2\eta
$$ 
and furthermore 
$$
\left\|\sum_{i,j=1}^{n}a_ixa_j^* \otimes b_ib_j^* - 1 \otimes x \right\|
, \,
\left\|\sum_{i,j=1}^{n}a_i^*a_j \otimes b_i^*xb_j - x \otimes 1 \right\|
 < 3\eta
$$

 Choose $W$ and $\delta$ as in Lemma~\ref{Lemma: stability} for the four expressions above and with $2\eta , \; 3\eta$ instead of $\eps$. We may assume without loss of generality that $\delta<\eta/3n$ and that $W$ contains $a_1,\ldots,a_n,b_1,\ldots,b_n$.

Choose a unital embedding $\varphi:\Zh \to \A$ such that $\|[\varphi(x),y]\|<\delta$ for all $x \in G \cup W$ and $y \in \alpha(\F) \cup \F \cup \alpha^{-1}(\F)$. Notice that $\|[\alpha \circ \varphi(x),y]\|<\delta$ for all $x \in G \cup W$ and $y \in \F$.

Now, choose a unital embedding $\psi:\Zh \to \A$ such that $\|[\psi(x),y]\|<\delta$ for all $x \in W$, $y \in \F \cup \alpha^{-1}(\F) \cup \varphi(G \cup W) \cup \alpha \circ \varphi(G \cup W)$.

Set 
$$v_1 = \sum_{i=1}^n \varphi(a_i)\psi(b_i) \; , \; v_2 = \sum_{i=1}^n \alpha(\varphi(a_i))\psi(b_i) \; .
$$

Now, for any $x \in \F$ we have that 
$$
\|v_1\varphi(x)v_1^* - \psi(x)\|<3\eta
$$ 
and 
$$
\|v_2^*\psi(x)v_2 - \alpha(\varphi(x))\|<3\eta \; .
$$

Notice that for all $a \in \F \cup \alpha(\F)$, we have that $\|[a,v_i]\| \leq 2n\delta < \eta$, $i=1,2$.

We also have that $\|v_i^*v_i - 1\| , \|v_iv_i^{*} - 1\|<2\eta$ for $i=1,2$. We may thus find unitaries $u_1,u_2$ such that $\|u_i-v_i\|<\gamma/10$. 
Note that  $\|[a,u_1]\| \leq \eta + 2\gamma/10$ for all $a \in \F \cup \alpha^{-1}(\F)$ and $\|[a,u_2]\| \leq \eta + 2\gamma/10$ for all $a \in \F \cup \alpha(\F)$.

Let $z = u_1^*u_2$. We note that if $x \in G$ then 
$$
\| z\alpha(\varphi(x))z^* - \varphi(x) \| \leq 4\gamma/10 + \| v_1^*v_2 \alpha(\varphi(x))v_2^*v_1 - \varphi(x) \| \leq 4\gamma/10 + 6\eta
$$ 
 
Note furthermore that if $a \in  \F \cup \alpha(\F)$ then $\|[z,a]\|< 2\eta + 4\gamma/10$.

We now define 
$$
\beta(a) = \textrm{Ad}(z) \circ \alpha \circ \textrm{Ad}(\varphi(u))  .
$$
So, for all $a\in \F$, we have that 
$$
\|\beta(a)-\alpha(a)\| =\| z \alpha(\varphi(u)a\varphi(u^*))z^* - \alpha(a)\| \leq \delta + 2\eta + 4\gamma/10 <\gamma
$$
and 
$$
\|\beta^{-1}(a)-\alpha^{-1}(a)\| = \| \varphi(u^*) \alpha^{-1}(z^*az)\varphi(u) - \alpha^{-1}(a)\| \leq \delta+ 2\eta + 4\gamma/10 < \gamma.
$$
It remains to check that $\beta \in V_{p,\F,\eps}$. We claim that the elements $\varphi(f_i)$, $\varphi(g_i)$ have the required properties. 
The facts that they are orthogonal to each other, add up to 1, $\eps$-commute with the elements of $\F$ and with each other follow immediately from the requirements we imposed on $\varphi$. It thus remains to check that they are almost permuted by $\beta$. Indeed, 
$$
\beta(\varphi(f_i)) 
= 
z \alpha(\varphi(u)\varphi(f_i) \varphi(u^*))z^*
=
z \alpha(\varphi(f_{i+1}))z^*
$$
and therefore 
$$
\| \beta(\varphi(f_i)) - \varphi(f_{i+1}) \| = \|z \alpha(\varphi(f_{i+1}))z^* - \varphi(f_{i+1})\| \leq 4\gamma/10 + 6\eta < \eps
$$
as required.
\end{proof}

\begin{Rmk}
A modification of this argument can be used to show the following. Let $\D$ be a strongly self-absorbing $C^*$-algebra, and let $\A$ be a unital $\D$-stable $C^*$-algebra, then for a dense $G_{\delta}$ set of automorphisms of $\A$ we have that $\A \rtimes_{\alpha} \Z$ is $\D$-stable as well.  This can be done by showing that for any given $\eps>0$ and finite sets $\F \subset \A$, $G\subset \D$, the set of automorphisms $\alpha$ such that there exists unital homomorphism $\varphi:\D \to \A$ such that $\|[\varphi(x),a]\|<\eps$ for all $x \in G$, $a \in \F$ and furthermore $\|\alpha(\varphi(x)) - \varphi(x)\|<\eps$ for all $x \in G$ is a dense open set. We omit the proof.
\end{Rmk}

\section{Permanence of finite nuclear dimension}

\noindent
In this section we show that forming a crossed product by an automorphism with finite Rokhlin dimension preserves finiteness of nuclear dimension.

\label{section:permanence-nd}

\begin{Thm}
Let $\A$ be a separable unital $C^*$-algebra of finite nuclear dimension and $\alpha \in \aut(\A)$ an automorphism with finite Rokhlin dimension. Then $\A \rtimes_\alpha \Z$ has finite nuclear dimension with 
$$\nd (\A \rtimes_{\alpha} \Z ) \leq 4(\dimrokh(\alpha)+1)(\nd (\A) + 1) -1$$
\end{Thm}

\begin{proof}
Denote the nuclear dimension of $\A$ by $N$ and denote $d = 2\dimrokh(\alpha)+1$. By Proposition~\ref{prop:single-Rokhlin-tower} and the subsequent remark, we have that $\dimrokhsingle(\alpha) \leq d$ and furthermore the towers can all be chosen to be of the same height. (The fact that the towers can be chosen to be of the same height is not important for the proof, but simplifies notation a bit.) Let $\F \subseteq \A \rtimes_{\alpha} \Z$ be a given finite set.
We need to construct a piecewise contractive c.p.\ approximation $( \mathcal{F} , \Phi, \Psi)$ which is 
$[2(d+1)(N + 1) -1]$-decomposable and of 
tolerance $\eps$ on $\F$. %In fact it suffices to show that 
%$\|  \Phi \circ \Psi (x) -x \| < \gamma (\eps)$ for all $x \in \F$, where $\gamma : \R_+ \to \R_+$ 
%is a function such that $\gamma (\eps) \to 0$ as $\eps \to 0$.  

Recall that $\A  \rtimes_\alpha \Z \hookrightarrow B(\ell^2(\Z,H))$ is generated by a copy of $\A$ 
and a unitary $u$ acting on $\ell^2(\Z,H)$, cf.\ the appendix. We may and shall assume that $\F$ consists 
of contractions all lying in the algebraic crossed product, i.e.\;there exists $q \in \N$ such that all
elements of $\F$ are of the form $x=\sum_{i=-q}^q x(i) u^i$, where $x(i) \in \A$ are coefficients 
uniquely determined by $x$. Let $\tilde{\F} \subseteq \A$ be the finite set of all such coefficients 
of the elements in $\F$. Let $p \in \N$ be a positive integer (much) larger than $q$ 
to be specified later. We shall furthermore require that $p$ is even to slightly simplify notation.

Let $Q$  be the projection onto the subspace 
$\ell^2(\{0, \ldots ,p-1\}, H)  \subseteq \ell^2(\Z,H)$ 
so that $x \mapsto  Q x Q$ 
is a u.c.p. map from  
$\A  \rtimes_\alpha \Z  $ to $M_{p}(\A)$ indexed by $0, \ldots ,p-1$. 

Define  decay factors 
$$
d_n = 1-\frac{|(p-1)/2 - n|}{(p-1)/2}, \; \; n=0,\ldots,p-1,
$$
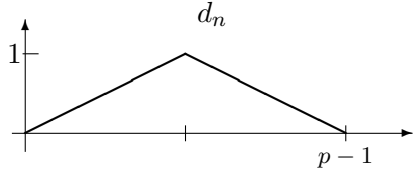
\begin{center}
\begin{picture}(230,65)
 \put(0,10){\vector(1,0){150}}
 \put(5,3){\vector(0,1){50}}
 \put(65,7){\line(0,1){6}}
 \put(125,7){\line(0,1){6}}
 \put(4,40){\line(1,0){6}}
\thicklines
 \put(5,10){\line(2,1){60}}
 \put(65,40){\line(2,-1){60}}
 \put(1,40){\makebox(0,0){$1$}}
 \put(125,-3){\makebox(0,0)[b]{\footnotesize $p-1$\normalsize}}
 \put(75,55){\makebox(0,0){$d_n$}}
\end{picture}
\end{center}

We observe that $|\sqrt{d_i}-\sqrt{d_j}| \leq \frac{2|i-j|}{p-1}$. Define $D \in M_p$ to be $\textrm{diag}(d_0,d_1,\ldots,d_{p-1})$. Define $\mu : \A  \rtimes_\alpha \Z \to M_p(\A)$ 
to be the c.p.c.\ map given by 
$$
\mu(x)=\sqrt{D}QxQ\sqrt{D} .
$$ 
Notice that 
\begin{eqnarray*}
\| [\sqrt{D},Q au^m Q] \| &=& \left\| \sum_{i=m}^{p-1}(\sqrt{d_i} - \sqrt{d_{i-m}}) 
e_{i,i-m} \otimes \alpha^{-i}(a) \right\| \\
& \leq &
\max _i(|\sqrt{d_i} - \sqrt{d_{i-m}}|) \|a\| \\
& \leq &
\sqrt{\frac{2|m|}{p-1}} \cdot  \|a\| \leq \sqrt{\frac{2q}{p-1}} \cdot  \|a\|  \; ,
\end{eqnarray*}
where $a \in \A$ and $|m| \leq q$.  In fact these estimates hold for 
all bounded operator matrices $x=[x_{i,j}]$ of width $q$ (i.e.\;$x_{i,j} \neq 0$ implies $|i-j| \leq q$). It follows that  
for $|m| \leq q$
$$
\|\mu(au^m) - DQau^mQ\| \leq \sqrt{\frac{2q}{p-1}} \cdot  \|a\|
$$
Next let $\tilde{\F}_1:= \tilde{\F} \cup \alpha^{-1}(\tilde{\F}) \cup \ldots \cup \alpha^{-p}(\tilde{\F})$ 
and let $(\mathcal{F} ,\psi,\varphi)$ be a piecewise contractive 
$N$-decomposable c.p.\ approximation, 
$$\psi : \A \to \mathcal{F}  =\mathcal{F}  ^{(0)} \oplus \ldots \oplus \mathcal{F}  ^{(N)},$$ 
$$\varphi : \mathcal{F}  = \mathcal{F}  ^{(0)} \oplus \ldots \oplus \mathcal{F}  ^{(N)} \to \A, $$ 
where $\varphi^{(i)}=\varphi | \mathcal{F}  ^{(i)}$ is an order zero contraction for every $i$ and 
$\| \varphi (\psi (x)) -x \| < \eps$ for all $x \in \tilde{\F}_1$. Let $\psi^{(i)}$ denote the $i$th component of $\psi$ and consider
$$
 \tilde\psi := \textup{id}_{ M_p} \otimes \psi : M_p(\A) \to  M_p(\mathcal{F}  ^{(0)}) \oplus \ldots \oplus M_p(\mathcal{F}  ^{(N)}) = M_p(\mathcal{F}  )
$$
and
$$
 \tilde\varphi := \textup{id}_{ M_p} \otimes \varphi : M_p(\mathcal{F}  ) = M_{p}( \mathcal{F}  ^{(0)}) \oplus \ldots \oplus M_{p}(\mathcal{F}  ^{(N)}) \to  M_p(\A)
$$ 
with components  $\tilde{\psi}^{(i)}= \textup{id}_{ M_p} \otimes \psi^{(i)}$, 
$\tilde{\varphi}^{(i)}= \textup{id}_{ M_p} \otimes \varphi^{(i)}$. Then one checks that 
$\tilde{\varphi}^{(i)}:M_p(\mathcal{F}  ^{(i)}) \to M_p(\A)$  are 
order zero contractions so that $(\tilde{\psi} ,\tilde{\varphi})$ is a piecewise contractive $N$-decomposable approximation 
for $M_p(\A)$ of tolerance $p^2 \eps $  on $M_p(\tilde\F_{1})$.

Let now $B_\mathcal{F}  $ be the closed unit ball in $\mathcal{F}  $ (which is norm compact) and define the norm compact subset 
$K=\bigcup_{j=0}^p \bigcup_{i=0}^N \alpha^{-j}(\varphi^{(i)}(B_\mathcal{F}  ))$ of $\A$. 

Let 
$(f^{(l)}_{i})_{i=0,\ldots,p-1}$ be single Rokhlin towers with respect to the given $\eps$ and the compact set $K$ (we note that one may obviously replace the finite set in the definition by a compact set). For $j > p-1$ we understand $f^{(l)}_j$ to mean $f^{(l)}_{j (\textrm{mod }p)}$.
Define $\rho^{(l)}_0 , \rho^{(l)}_1 : M_{p} (\A) \to \A\rtimes_{\alpha}\Z $ by 
$$
\rho^{(l)}_0 (e_{i,j} \otimes a) = {f_{i}^{(l)}}^{1/2}  u^i a u^{-j} {f_{j}^{(l)}}^{1/2} 
\;\; , \;\;
\rho^{(l)}_1 (e_{i,j} \otimes a) = ({f_{p/2+i}^{(l)}})^{1/2}  u^i a u^{-j}
({f_{p/2+j}^{(l)}})^{1/2} \; .
$$

One checks that $\rho^{(l)}_0 $ is approximately order zero i.e.\ denoting $ f^{(l)}= \sum_i f_{i}^{(l)}$ we have:
$$
\| \rho^{(l)}_0(e_{i,j} \otimes a) \rho^{(l)}_0(e_{k,m} \otimes b) - 
f^{(l)} \rho^{(l)}_0((e_{i,j} \otimes a)(e_{k,m} \otimes b)) \| < cp\eps
$$
for a constant $c>0$ which depends only on $\eps$ and for all $a,b \in K$. The same applies for
 $\rho^{(l)}_1$. This is easily verified by approximating the square root function by a polynomial vanishing at $0$, and we omit the calculation. Similarly, we see that there is a constant $c'$ such that 
$$
\|\rho^{(l)}_0 (e_{i,j} \otimes a) - {f_{i}^{(l)}} u^i a u^{-j} \| < c'\eps
$$
for all $a \in \tilde{\F}_{1}$, and a similar estimate holds for $\rho^{(l)}_1$.

Let $x = au^m$ for $|m|\leq q$, $a \in \tilde{\F}$. For $m \ge 0$ denote $y = QxQ = \sum_{i=m}^{p-1}e_{i,i-m}\otimes \alpha^{-i}(a)$; the case $m < 0$ is handled analogously and we skip the details. 
Then we have
\begin{gather*}
 \rho^{(l)}_0 (Dy) + \rho^{(l)}_1 (Dy)   =   \\ 
  \sum_{i=m}^{p-1}d_i  f_i^{{(l)}^{1/2}} u^i \alpha^{-i}(a) u^{m-i}  f_{i-m}^{{(l)}^{1/2}} 
+
\sum_{i=m}^{p-1}d_i f_{p/2+i}^{{(l)}^{1/2}} u^i \alpha^{-i}(a) u^{m-i}  f_{p/2+i-m}^{{(l)}^{1/2}} 
  \approx_{2pc'\eps}  \\
   \sum_{i=m}^{p-1}d_i  (f_i^{(l)} + f_{p/2+i}^{(l)})  au^{m} 
\end{gather*}
Now, notice that $\sum_{i=0}^{p-1}d_i  (f_i^{(l)} + f_{p/2+i}^{(l)})   = f^{(l)}$, and 
$$
\left \|\sum_{i=0}^{m-1}d_i  (f_i^{(l)} + f_{p/2+i}^{(l)}) \right \| \leq 2\sum_{i=0}^{q-1}d_i = \frac{q(q-1)}{(p-1)/2} \; .
$$ 
It follows that $\|\rho^{(l)}_0 (Dy) + \rho^{(l)}_1 (Dy) - f^{(l)}au^m \| \leq 2pc'\eps + \frac{q(q-1)}{(p-1)/2}$, and therefore 
$$
\|\rho^{(l)}_0 \circ \mu(y) + \rho^{(l)}_1\circ \mu(y) - f^{(l)}au^m \| \leq 2pc'\eps + \frac{q(q-1)}{(p-1)/2} + 2\sqrt{\frac{2q}{p-1}}
$$
Thus, given $\eta>0$ we could first choose $p$ large enough and then $\eps>0$ small enough so that 
$$
\|\rho^{(l)}_0 \circ \mu(au^m) + \rho^{(l)}_1\circ \mu(au^m) - f^{(l)}au^m \| \leq \eta
$$
for all $a \in \F$ and all $|m|\leq q$.
Consider the maps 
$$
\bar{\psi} : M_p (\A) \to  \bigoplus_{j=0}^{2d+1} M_p (\mathcal{F}  )
$$
given by $\bar{\psi} (x) = \tilde\psi(x) \oplus \tilde\psi(x) \oplus \ldots \tilde\psi(x)$ and
$$
\bar{\varphi} :   \bigoplus_{j=0}^{2d+1} M_p (\mathcal{F}  )  \to  \bigoplus_{j=0}^{2d+1} M_p (\A)
$$
given by $\bar{\varphi} =  \bigoplus_{j=0}^{2d+1} \tilde\varphi$. 

Now define $\rho:  \bigoplus_{j=0}^{2d+1} M_p (\A) \to \A \rtimes_{\alpha} \Z$ by
$$
\rho(x_0,x_1,\ldots,x_{2d+1}) = \sum_{j=0}^{d-1}\sum_{k=0}^1 \rho_k^{(j)}(x_{2j+i})
$$
With $a,m$ and $\eta$ as above, we have that 
$$
\left\|\rho(\mu(au^m),\mu(au^m),\ldots,\mu(au^m)) - \sum_{l=0}^{d}f^{(l)} au^m \right\| \leq d\eta
$$ and therefore 
$$
\|\rho(\mu(au^m),\mu(au^m),\ldots,\mu(au^m)) - au^m\|\leq d\eta + \eps \; .
$$

To summarize, we constructed the following maps.
$$
\xymatrix{
\A\rtimes_{\alpha}\Z \ar[dr]^{\mu} \ar@/_2pc/[ddrr]_{\Psi} &   & &  &  \A\rtimes_{\alpha}\Z \\ 
  &  M_p(\A)    \ar[dr]^{\bar{\psi}} &  &  \bigoplus_{j=0}^{2d+1} M_p(\A)   \ar[ur]^{\rho} &   \\
  &  & \bigoplus_{j=0}^{2d+1} M_p (\mathcal{F}  )  \ar[ur]^{\bar{\varphi}}  \ar@/_2pc/[uurr]_{\Phi} & &   }
$$

We notice that $\rho_k^{(j)} \circ \tilde{\varphi}^{(i)}|_{M_p(\mathcal{F}^{(i)})}$ is approximately order zero, i.e.\ for any $a,b \in M_p(\mathcal{F}^{(i)})$ of norm at most 1, we have that 
$$\|\rho_k^{(j)} \circ \tilde{\varphi}^{(i)}(a)\rho_k^{(j)} \circ \tilde{\varphi}^{(i)}(b) - f^{(j)}\rho_k^{(j)} (\tilde{\varphi}^{(i)}(1_{M_{p}(\mathcal{F}^{(i)})}) \tilde{\varphi}^{(i)}(ab))\|<cp\eps.$$ 
By stability of order zero maps, given $\eta>0$ one can choose $\eps>0$ small enough so that for this choice of $\eps$, there exists a contractive order zero map $\zeta:M_p(\mathcal{F}^{(i)}) \to \A \rtimes_{\alpha} \Z$ such that $\|\zeta - \rho_k^{(j)} \circ \tilde{\varphi}^{(i)}\|<\eta$. Putting together such approximating order zero maps $\zeta$ for the various maps $ \rho_k^{(j)} \circ \tilde{\varphi}^{(i)}$, we see that there is a $(2(d+1)(N+1)-1)$-decomposable map $\Phi': \bigoplus_{j=0}^{2d+1} M_p (\mathcal{F}  ) \to \A \rtimes_{\alpha}\Z$ such that $\|\Phi' -\Phi\| \leq 2(d+1)(N+1)\eta$. 

One obtains that 
$$
\xymatrix{
\A\rtimes_{\alpha}\Z \ar[dr]^{\Psi}&    &  \A\rtimes_{\alpha}\Z \\ 
  &   \bigoplus_{j=0}^{2d+1} M_p (\mathcal{F}  )  \ar[ur]^{\Phi'}  &   }
$$
is a  $(2(d+1)(N+1)-1)$-decomposable approximation for $\F$ to within $C\eta$ for some constant $C$ which depends only on $d$ and $N$. This shows that $\nd(\A \rtimes_{\alpha}\Z) \leq 2(d+1)(N+1)-1$, as required.
\end{proof}

\section{Permanence of $\mathcal{Z}$-stability}
\label{section:permanence-Z}

\noindent
 The purpose of this section is to show that 
if $\A$ is a unital separable $\mathcal{Z}$-stable $C^{*}$-algebra, then so is any crossed product by an automorphism which has finite Rokhlin dimension with commuting towers. As shown above, this property is generic and we do not know whether there is an automorphism which has finite Rokhlin dimension but without commuting towers. Our prefered criterion for $\mathcal{Z}$-stability will be Proposition~\ref{Z-stable-crit} of the appendix.

We begin with a simple preliminary lemma, whose proof is straightforward and left to the reader.

\begin{Lemma}
\label{lemma:generators-almost-fixed}
Let $\B$ be a $C^*$-algebra. Suppose $\B$ is generated by a compact set $Y \subseteq \B$. For any $\eps>0$ and any compact subset $F   \subseteq \B$ there is an $\eps'>0$ which satisfies the following. For any $C^*$-algebra $\A$, any automorphism $\alpha$ of $\A$ and any homomorphism $\pi:\B \to \A$, if $\|\alpha(\pi(y))-\pi(y)\|<\eps'$ for any $y \in Y$ then $\|\alpha(\pi(x))-\pi(x)\|<\eps$ for any $x \in F  $.
\end{Lemma}

The following lemma is also easy to verify and we leave its proof to the reader. We denote by $\A^+$ the unitization of $\A$ such that $\A$ is a codimension $1$ ideal in $\A^+$ (that is, if $\A$ is unital then $\A^+ \cong \A \oplus \C$).
\begin{Lemma}
\label{lemma:univ-prop-two-maps}
Let $\A,\B$ be two $C^*$-algebras. Let $\D$ be the kernel of the canonical map $\A^+ \otimes_{\max} \B^+ \to \C$, and let us denote by $\iota_{\A},\iota_{\B}$ the following homomorphisms. 
$$
\iota_{\A}(a) = a \otimes 1  \;\;\;\; \iota_{\B}(b) = 1 \otimes b
$$
Then $\D$ has the following universal property. For any $C^*$-algebra $\E$ and any two homomorphisms $\gamma_{\A}:\A \to \E$, $\gamma_{\B}:\B \to \E$ with commuting images there is a unique homomorphism $\theta: \D \to \E$ such that the following diagram commutes.

$$\xymatrix{
\A \ar[r]^{\iota_{\A}} \ar[dr]_{\gamma_{\A}} & \D \ar[d]_{\theta} & \ar[dl]^{\gamma_{\B}} \ar[l]_{\iota_{\B}}  \B \\
& \E &
}$$ 
\end{Lemma}

\begin{Lemma}
\label{lemma:univ-alg-of-order-0-maps}
Let $\D_n^{(k)} = \ker\left ( (CM_n^+)^{\otimes k} \to \C \right )$. For any positive contraction $h \in Z(\D_n^{(k)}) \cong C_0([0,1]^k \setminus\{(0,0,\ldots,0)\})$ there exists an order zero map $\theta:M_n \to \D_n^{(k)}$ such that $\theta(1)(x) = h(x)$ for all $x \in [0,1]^k$. 
\end{Lemma}

\begin{proof}
For $t \in [0,1]$, let 
$$
\E_t = \left \{ \begin{matrix} M_n & \mid & t>0 \\
\C \cdot 1 \subseteq M_n & \mid & t=0
					\end{matrix} \right .
$$
For $\vec{t} = (t_1,\ldots,t_k) \in [0,1]^k$, we let 
$$
\E_{\vec{t}} = \E_{t_1} \otimes \E_{t_2} \otimes \ldots \otimes \E_{t_k}
$$
We observe that 
$$
\D_n^{(k)} \cong \left \{ f \in C_0([0,1]^k\setminus \{(0,0,\ldots,0)\},M_n^{\otimes k}) \mid 
						f(\vec{t}) \in \E_{\vec{t}}
				\right \}
$$
It is easy to construct a unital homomorphism $\pi:M_n \to M(\D_n^{(k)})$, and it is readily seen that $\theta(x)(\vec{t}) = h(\vec{t})\cdot \theta(x)(\vec{t})$ satisfies the desired properties.
\end{proof}

The following is a simple modification of Lemma 2.4 from \cite{hirshberg-winter}. While that lemma can be generalized to actions of non-discrete groups in a straightforward way (as was done in \cite{hirshberg-winter}), we state it here just for actions of discrete groups to avoid notation which we do not need in this paper. It should be pointed out that our argument is essentially the same as that of  \cite[Theorem~4.9]{MatSat:strongly-outerII}. We reproduce the full proof here for the convenience of the reader.

We denote $\A_{\infty} = \ell^{\infty}(\N,\A)/c_0(\N,\A)$, with $\A$ embedded as  the subalgebra of constant sequences in $\A_{\infty}$. We use a similar idea to one which was used in Proposition 2.2 from \cite{toms-winter-ash}.

If $\alpha:G \to \aut(\A)$ is an action then we have naturally induced actions of $G$ on $\A_{\infty}$ and $\A_{\infty} \cap \A'$. We denote those actions by $\bar{\alpha}$.

\begin{Lemma}
\label{lemma:enough-to-embed-building-block}
Let $\A,\B$ be a unital separable $C^*$-algebras. Let $G$ be a discrete countable group with an action $\alpha:G \to \aut(\A)$. Suppose that $\B_n$ is a sequence of unital nuclear subalgebras (with a common unit) of $\B$ with dense union. Suppose that for any $n$, any finite subset $F \subseteq \B_n$,  any $\eps>0$ and any finite set $G_0 \subseteq G$ there is a unital homomorphism $\gamma:\B_n \to \A_{\infty}\cap \A '$ such that 
$$
\|\bar{\alpha}_g(\gamma(a)) - \gamma(a)\|<\eps
$$ 
for all $a \in F$ and all $g \in G_0$, then there is a unital homomorphism from $\B$ into the fixed point subalgebra of $ \A_{\infty}\cap \A '$. If $\B_n$ is finitely generated, then it suffices to check this for a generating set $F$. 

If $\B$ is furthermore strongly-self absorbing then it follows that the maximal (hence any) crossed product absorbs $\B$ tensorially.
\end{Lemma}

\begin{proof}
 To verify the properties of Lemma 2.4 from \cite{hirshberg-winter}, it suffices to show that for any finite set $F \subset \B$, any finite subset $G_0 \subseteq G$  and any $\eps>0$ there is a unital c.p. map $\varphi:\B \to \A_{\infty} \cap \A '$ such that 
\begin{enumerate}
\item $\|\bar{\alpha}_g (\varphi(b)) - \varphi(b)\| <\eps$ for all $g \in G_0$, $b \in F$.
\item $\|\varphi (x)\varphi(y) - \varphi(xy)\|<\eps$ for all $x,y \in F$.
\end{enumerate}
Doing a small perturbation of the elements of $F$ if need be, we may assume without loss of generality that $F \subseteq \B_n$ for a sufficiently large $n$, and that all the elements of $F$ have norm at most 1, and that $1 \in F$ (so that $F \cdot F \supseteq F$ - this is just for notational convenience).
Fix such an $n$. Since $\B_n$ is nuclear, one can find a finite dimensional algebra $\E$ and unital c.p. maps 
$$
\xymatrix{\B_n \ar[r]^{\psi} & \E \ar[r]^{\theta} & \B_n}
$$ such that $\|\theta \circ \psi(x) - x\|<\eps/6$ for all $x \in F \cdot F$. It follows that 
$$
\|\theta\circ \psi(x)\cdot \theta\circ \psi(y) - \theta\circ \psi(xy)\|<\eps/2
$$ 
for all $x,y \in F$. By the Arveson extension theorem, we can extend $\psi$ to a unital c.p. map $\tilde{\psi}:\B \to \E$. Choose a homomorphism $\gamma$ as in the statement of the lemma, with $\eps/2$ instead of $\eps$. Define 
$$
\varphi = \gamma \circ \theta \circ \tilde{\psi}
$$
and $\varphi$ satisfies the required conditions. 
It is straightforward to check that if one can find such a $\varphi$ for a generating set then one could find it for any other finite set.

If $\B$ is strongly self-absorbing then the fact that $\A \rtimes_{\alpha} G$ absorbs $\B$ tensorially now follows from the results in \cite{hirshberg-winter}.
\end{proof}

The previous lemma, together with the characterization of prime dimension drop algebras from \cite{rordam-winter}  immediately gives the following; cf.\ \cite[Theorem~4.9]{MatSat:strongly-outerII}.
\begin{Cor}
\label{cor:enough-to-embed-order-zero-pair}
Let $\A$ be a unital separable $C^*$-algebra. Let $G$ be a discrete countable group with an action $\alpha:G \to \aut(A)$. Suppose that for any positive integer $n$, any $\eps>0$ and any finite subset $G_0 \subseteq G$ there are order zero maps $\theta:M_n \to \A_{\infty}\cap \A '$,  $\eta:M_{n+1} \to \A_{\infty}\cap \A '$ with commuting ranges such that 
$$\theta(1)+\eta(1) = 1$$ and such that 
$$
\|\bar{\alpha}_g(\theta(x)) - \theta(x)\|<\eps
$$ 
and 
$$
\|\bar{\alpha}_g(\eta(x)) - \eta(x)\|<\eps
$$ 
for any $x$ in the unit ball of $M_n$ or $M_{n+1}$, respectively, and any $g \in G_{0}$, then $\A\rtimes_{\alpha} G$ is $\Zh$-stable.
\end{Cor}

\begin{Rmk}
In the previous lemma and corollary, if $G$ is generated by a subset $\Gamma$ then it is sufficient to consider finite subsets $G_0 \subseteq \Gamma$ rather than all finite subsets in $G$. In particular, for actions of $\Z$ we will consider a single generator.
\end{Rmk}

\begin{Lemma}
\label{lemma:enough-to-embed-several-order-zero-pairs}
Let $\A$ be a unital separable $C^*$-algebra. Let $G$ be a discrete countable group with an action $\alpha:G \to \aut(A)$. Let $d$ be a non-negative integer. Suppose that for any positive integer $n$, any $\eps>0$ and any finite subset $G_0 \subseteq G$ there are order zero maps $\theta_0,\ldots,\theta_d:M_n \to \A_{\infty}\cap \A '$,  $\eta_0,\ldots,\eta_d:M_{n+1} \to \A_{\infty}\cap \A '$ with commuting ranges such that 
for all $g \in G_0$, $k=0,1,\ldots,d$ and all $x$ in the unit ball of $M_n$ or $M_{n+1}$ respectively we have that 
$$
\|\bar{\alpha}_g(\theta_k(x)) - \theta_k(x)\|<\eps \; ,
$$ 
$$
\|\bar{\alpha}_g(\eta_k(x)) - \eta_k(x)\|<\eps
$$ 
and
$$
\sum_{k=0}^d \theta_k(1) + \eta_k(1) = 1
$$
then $\A\rtimes_{\alpha}G$ is $\Zh$-stable.
\end{Lemma}

\begin{proof}	
Let $\D_n^{(d+1)}, \D_{n+1}^{(d+1)}$ be as in Lemma \ref{lemma:univ-alg-of-order-0-maps}. 

Denote by $\zeta_j:CM_n \to \D_n^{(d+1)}$ the $j$'th coordinate embedding map 
$$
\zeta_j(x) = 1 \otimes 1 \otimes \ldots \otimes x \otimes \ldots \otimes 1
$$
Let $\beta:M_n \to CM_n$ be the canonical order zero map given by 
$$
\beta(a)(t) = t \cdot a,  \;\;\;\;\;\;\; a \in M_n, t \in [0,1]
$$

Denote $h = \sum_{k=0}^d \zeta_k(\beta(1))$. Note that $h$ is a strictly positive element of the center of $\D_n^{(d+1)}$. 
Denote by $M_n^1$ the closed unit ball of $M_n$, and set 
$$Y = \bigcup_{k=0}^d \zeta_k(\beta(M_n^1))$$ 
Notice that $Y$ is compact and generates $\D_n^{(d+1)}$.

 Use Lemma~\ref{lemma:univ-alg-of-order-0-maps} to find an order zero map $\mu:M_n \to \D_n^{(d+1)}$ such that $\mu(1) = h$. (Lemma~\ref{lemma:univ-alg-of-order-0-maps} was only formulated for contractions, which $h$ is not -- but we can simply replace $h$ by $\|h\|^{-1} h$, then apply the lemma, and multiply the resulting order zero map by $\|h\|$.)    Denote $F   = \mu(M_n^1)$.

We repeat the same procedure for $\D_{n+1}^{(d+1)}$, denoting the resulting maps, sets and elements by $\zeta_j'$, $\beta'$, $h'$, $Y'$, $\mu'$ and $F'$.

Choose $\eps'$ as in Lemma \ref{lemma:generators-almost-fixed} with respect to the compact set of generators $Y$ and the compact set $F$ and with respect to the 
compact set of generators $Y'$ and the compact set $F'$ (take the least of the two). 

Let 
$$\pi:\D_n^{(d+1)} \to \A_{\infty}\cap \A '$$ 
be the unique homomorphism which satisfies that $\pi \circ \zeta_k \circ \beta = \theta_k$ for $k=0,1,\ldots d$. Notice that
$$\pi(h) = \sum_{k=0}^d\theta_k(1)$$
(and that $\pi(h)$ in fact is a contraction). 
We have then that $\pi \circ \mu:M_n \to \A_{\infty} \cap \A'$ is an order zero map, that 
$$
\pi \circ \mu (1) = \sum_{k=0}^d \pi \circ \zeta_k \circ \beta(1)
$$
and that for any $x \in M_n^1$ and any $g \in G_0$ we have
$$
\|\bar{\alpha}_g(\pi \circ \mu (x)) - \pi \circ \mu(x)\|<\eps
$$
We similarly obtain a homomorphism $$\pi':\D_{n+1}^{(d+1)} \to \A_{\infty}\cap \A '$$ 
whose range commutes with that of $\pi$, such that 
$$
\pi'(h') = \sum_{k=0}^{d}\eta_k(1)
$$ 
and such that $\pi' \circ \mu'$ satisfies the analogous properties to that of $\pi \circ \mu$.

Thus, $\pi \circ \mu (1) + \pi' \circ \mu'(1) = 1$, and those two order zero maps satisfy the conditions of Corollary \ref{cor:enough-to-embed-order-zero-pair}. Therefore $\A\rtimes_{\alpha}G$ is $\Zh$-stable, 
as required.
\end{proof}

The main theorem of this section is a partial generalization of Theorem 4.4 
from \cite{hirshberg-winter} (that theorem is for the Rokhlin property, but it applies to absorption of general strongly self-absorbing 
$C^*$-algebras and not just the Jiang-Su algebra).

For the proof, we shall implicitly use the following immediate observation: by choosing a sequence of Rokhlin tower elements, one can view them as sitting in the central sequence algebra $\A_{\infty} \cap \A'$, in which case the approximate properties in the definition of Rokhlin dimension hold exactly. 

The reason why we have to assume the Rokhlin property with commuting towers is that we want to use Lemma~\ref{lemma:enough-to-embed-several-order-zero-pairs}, which in turn relies on Corollary~\ref{cor:enough-to-embed-order-zero-pair}. The latter required us to produce almost central and almost invariant copies of prime dimension drop intervals, which in turn are given by pairs of \emph{commuting} cones over $M_{n}$ and $M_{n+1}$, respectively. Lemma~\ref{lemma:enough-to-embed-several-order-zero-pairs} then showed that instead of finding pairs of such matrix cones it is in fact enough to find larger ensembles of mutually commuting matrix cones, as long as they add up to the identity in a suitable way. 

\begin{Thm}
\label{Thm:permanence-Z}
Let $\A$ be a separable unital $\Zh$-stable $C^*$-algebra, and let $\alpha$ be an automorphism of $\A$ with $\dimrokhct(\alpha)=d<\infty$, then the crossed product $\A \rtimes_{\alpha}\Z$ is $\Zh$-stable as well.
\end{Thm}

\begin{proof}
We establish that the conditions of Lemma \ref{lemma:enough-to-embed-several-order-zero-pairs} hold. 

Let $r$ be a given positive integer. Fix two order zero maps $\theta:M_r \to \Zh$, $\eta:M_{r+1} \to \Zh$ with commuting ranges such that 
$\theta(1)+\eta(1) = 1$. Let $K$ be the union of the images of the unit balls of $M_r$, $M_{r+1}$ under those maps.

We define unital homomorphisms $\iota_0,\ldots,\iota_d,\mu_0,\ldots,\mu_d:\Zh \to \A_{\infty} \cap \A '$ as follows. First fix 
$\iota_0:\Zh \to \A_{\infty} \cap \A '$. Proceeding inductively, we choose $\iota_k:\Zh \to \A_{\infty} \cap \A '$ such that its image furthermore commutes with  $\bar{\alpha}^j(\iota_i(\Zh))$ for all $j \in \Z$ and $i<k$ (this can be done by Lemma 4.5 of \cite{hirshberg-winter}). We define $\mu_0,\ldots,\mu_d$ in a similar way, such that the image of $\mu_k$ commutes with
 $\bar{\alpha}^j(\mu_i(\Zh))$ for all $j \in \Z$ and $i<k$ as well as 
  $\bar{\alpha}^j(\iota_i(\Zh))$ for all $j \in \Z$ and all $i=0,1,\ldots ,d$.

 Let $$\textstyle \B_k := C^*\left ( \mu_k(\Zh) \cup
\bigcup_{j=-\infty}^{\infty} \bar{\alpha}^j(\iota_k(\Zh)) \right ) $$ note that $\B_k \cong \B_k \otimes \Zh$, and that the elements of 
$\B_k$ commute with those of $\B_m$ for $k \neq m$.

Choose a unitary $w \in U(\Zh \otimes \Zh)$ such that
$$\|w(x\otimes 1)w^*-1 \otimes x\| < \frac{\eps}{4}$$ for all $x \in K$. Notice that the unitary group of the Jiang-Su algebra is connected. 
Thus, $w$ can be connected to $1$ via a rectifiable path.
Let $L$ be the length of such a path. Choose $n$ such that
$L\|x\|/n <\eps/8$ for all $x \in F$.

 Define homomorphisms $$\rho_k,\rho_k':\Zh \otimes \Zh \to
\B_k \;\;\;\;\; k=0,1,\ldots,d$$ by
$$\rho_k(x \otimes y) = \iota_k (x) \mu_k(y), \;\; \rho_k'(x\otimes y) =
\bar{\alpha}^n(\iota_k(x))\mu_k(y) \, .$$ Pick unitaries $1=w_0,w_1,\ldots,w_n = w \in
U_0(\Zh \otimes \Zh)$ such that $\|w_j-w_{j+1}\| \leq L/n$ for
$j=0,\ldots,n-1$. Now, for $k=0,\ldots,d$, let 
$$
u_j^{(k)} = \rho_k(w_j)^*\rho_k'(w_j)
$$ Note that
$$\|u_j^{(k)} - u_{j+1}^{(k)}\| \leq \frac{2L}{n}$$ for $j=0,\ldots,n-1$, that $u_0^{(k)} = 1$,
and that $$\|u_n^{(k)} \bar{\alpha}^n(\iota_k(x)) u_n^{(k) \, *} - \iota_k(x) \| <
\frac{\eps}{2}$$ for all $x\in K$.

Similarly, we choose unitaries $1=v_0^{(k)} ,v_1^{(k)},\ldots,v_{n+1}^{(k)} \in \B_k$
such that $$\|v_j^{(k)}-v_{j+1}^{(k)}\| \leq \frac{2L}{n+1}$$ for $j=0,\ldots,n$
and
$$\|v_{n+1}^{(k)}\bar{\alpha}^{n+1}(\iota_k(x))v_{n+1}^{(k)\, *} - \iota_k(x)\| < \frac{\eps}{2}$$
for all $x \in K$.

Let $\{f_{0,0}^{(l)},\ldots,f_{0,n-1}^{(l)}, f_{1,0}^{(l)},\ldots,f_{1,n}^{(l)} \mid l = 0,\ldots,d\} \in \A_{\infty} \cap \A '$ 
commuting Rokhlin elements in $\A_{\infty} \cap \A '$ which furthermore commute with $\B_0,\B_1,\ldots,\B_d$. 

Now, set
\begin{eqnarray*}
\theta_k (x)& = & \sum_{j=0}^{n-1}f_{0,j}^{(k)}\bar{\alpha}^{j-n}(u_j^{(k)})\bar{\alpha}^j
(\iota_k \circ \theta(x))\bar{\alpha}^{j-n}(u_j^{(k)\, *}) \\
&&+ 
\sum_{j=0}^{n}f_{1,j}^{(k)}\bar{\alpha}^{j-n}(v_j^{(k)})\bar{\alpha}^j
(\iota_k \circ \theta(x))\bar{\alpha}^{j-n}(v_j^{(k)\, *})
\end{eqnarray*}
and
\begin{eqnarray*}
\eta_k (x)& = & \sum_{j=0}^{n-1}f_{0,j}^{(k)}\bar{\alpha}^{j-n}(u_j^{(k)})\bar{\alpha}^j
(\iota_k \circ \eta(x))\bar{\alpha}^{j-n}(u_j^{(k)\, *}) \\
&&+ 
\sum_{j=0}^{n}f_{1,j}^{(k)}\bar{\alpha}^{j-n}(v_j^{(k)})\bar{\alpha}^j
(\iota_k \circ \eta(x))\bar{\alpha}^{j-n}(v_j^{(k)\, *}).
\end{eqnarray*}

We can check that 
$$
\|\bar{\alpha}(\theta_k(x)) - \theta_k(x)\|<\eps
$$
and 
$$
\|\bar{\alpha}(\eta_k(x)) - \eta_k(x)\|<\eps
$$
for all $x$ in the unit balls of $M_r$, $M_{r+1}$, respectively.  
Furthermore, 
$$
\theta_k(1) + \eta_k(1) = \sum_{j=0}^{n-1}f_{0,j}^{(k)} + \sum_{j=0}^{n}f_{1,j}^{(k)}
$$ 
Therefore those maps satisfy the conditions of Lemma~\ref{lemma:enough-to-embed-several-order-zero-pairs}.
\end{proof}

We can also obtain in a similar way the analogous result for actions of finite groups with finite Rokhlin dimension and commuting towers.

\begin{Thm}
\label{finite-action-Z-stable}
Let $\A$ be a separable unital $\Zh$-stable $C^*$-algebra, let $G$ be a finite group and let $\alpha:G \to \aut(\A)$ be an action with Rokhlin dimension $d<\infty$, such that furthermore the Rokhlin elements from Definition \ref{def: positive Rokhlin finite groups} can be chosen to commute with each other. Then the crossed product $\A \rtimes_{\alpha}G$ is $\Zh$-stable as well.
\end{Thm}

\begin{proof}
The proof is a simpler version of the proof of Theorem \ref{Thm:permanence-Z} - here we do not need the choice of correcting unitaries. 
We again establish that the conditions of Lemma \ref{lemma:enough-to-embed-several-order-zero-pairs} hold, and we start in a similar way.

Let $r$ be a given positive integer. Fix two order zero maps $\theta:M_r \to \Zh$, $\eta:M_{r+1} \to \Zh$ with commuting ranges such that 
$\theta(1)+\eta(1) = 1$. 
We define unital homomorphisms $\iota_0,\ldots,\iota_d:\Zh \to \A_{\infty} \cap \A '$ as follows. First fix 
$\iota_0:\Zh \to \A_{\infty} \cap \A '$. Proceeding inductively, we choose $\iota_k:\Zh \to \A_{\infty} \cap \A '$ such that its image furthermore commutes with  $\bar{\alpha}_g(\iota_i(\Zh))$ for all $g \in G$ and $i<k$. 

 Let $$\textstyle \B_k := C^*\left (
\bigcup_{g\in G} \bar{\alpha}_g(\iota_k(\Zh)) \right ); $$ note that the elements of 
$\B_k$ commute with those of $\B_m$ for $k \neq m$.

Let $\left( f^{(l)}_g \right)_{l=0,\ldots,d;\; g \in G} \subseteq \A_{\infty} \cap \A'$ be Rokhlin elements, and as in the statement of the theorem we assume that they all commute with each other, and which are furthermore chosen to commute with $\B_0,\B_1,\ldots,\B_d$. 

Define
$$
\theta_k (x)= \sum_{g \in G} f_{g}^{(k)}\bar{\alpha}_g
(\iota_k \circ \theta(x)) \;\; , \;\; 
\eta_k (x)= \sum_{g \in G} f_{g}^{(k)}\bar{\alpha}_g(\iota_k \circ \eta(x)).
$$
One checks that the images of those maps are fixed by the action of $G$ on $\A_{\infty} \cap \A'$, and 
satisfy the conditions of Corollary \ref{cor:enough-to-embed-order-zero-pair}.
\end{proof}

\section{Topological dynamics}
\label{section-irrational}

\noindent
In the initial version of this paper we showed that minimal homeomorphisms on finite dimensional compact metrizable spaces always have finite Rokhlin dimension:

\begin{Thm}
\label{main}
Let $(T,h)$ be a minimal dynamical system with $T$ compact and metrizable; suppose that  $\dim T \le n<\infty$ and let $\alpha$ be the induced action on $C(T)$. Then, 
\[
\dimrokh^{\mathrm{s}} (C(T),\alpha) \le 2 n+1.
\] 
\end{Thm}

The proof relied on a detailed analysis of the representation theory of certain large orbit breaking subalgebras of the transformation group $C^{*}$-algebra. After the paper appeared on the arXiv, Szab\'o  in \cite{Sza:dimnuc} generalized this to the case of aperiodic homeomorphisms. In fact, with essentially the same methods, he could push the result even to free and aperiodic $\mathbb{Z}^{d}$-actions.

Szab\'o's proof -- although by no means trivial -- is somewhat easier and also more conceptual than our original argument, which we therefore have removed from this version (even though at least it has the slight advantage over Szab\'o's proof that it provides particularly nice and fairly explicit approximations of large orbit breaking subalgebras of the crossed product). However, the reader might find the following short and direct direct proof of the fact that irrational rotations have Rokhlin dimension $1$ instructive.

\begin{Thm} \label{Thm:irrational-rotation}
Let $\alpha$ be the irrational rotation by $\theta$ on $C(\T)$, then $\dimrokhsingle(\alpha) = 1$.
\end{Thm}
\begin{Rmk}
It follows immediately that if $(X,h)$ is a dynamical system which has an irrational rotation as a factor, that is there is a commuting diagram 
$$
\xymatrix{
X \ar[r]^h \ar[d]_{\pi} & X \ar[d]^{\pi} \\
\T \ar[r]^{\rho} & \T
}
$$
then $\dimrokhsingle(C(X),\alpha) \leq 1$ as well.
\end{Rmk}

\begin{proof}[Proof of Theorem \ref{Thm:irrational-rotation}]
Given a prime number $p$ and $\eps>0$, we will exhibit positive $\{\tilde{f}_i,\tilde{g}_i\}_{i=0,1,\ldots,p-1}$ such that 
\begin{enumerate}
\item $\displaystyle \|\alpha(\tilde{f}_j) - \tilde{f}_{j+1\, \textrm{(mod)}\,p}\| \; , \; \|\alpha(\tilde{g}_j) - \tilde{g}_{j+1\, \textrm{(mod)}\,p}\| <\eps$
\item  The $\tilde{f}_j$'s are pairwise orthogonal and the $\tilde{g}_j$'s are pairwise orthogonal.
\item $\displaystyle \tilde{f}_0+\tilde{f}_1+\ldots+\tilde{f}_{p-1}+\tilde{g}_0+\tilde{g}_1+\ldots+\tilde{g}_{p-1}=1 $
\end{enumerate}
and that shows that $\alpha$ has Rokhlin dimension $1$ with single towers. We note that since $C(\T)$ has no nontrivial projections, $\alpha$ clearly cannot have Rokhlin dimension $0$.

We first recall some basic facts about continued fractions (see \cite{schmidt}, Chapter 1). If $t$ is an irrational number and if $m_j/n_j$ is the sequence of approximants for $t$, then $|t-\frac{m_j}{n_j}|<\frac{1}{n_j^2}$ for all $j$. Furthermore, the $m_j$'s satisfy a recursive formula of the form $m_{j+1} = a_j m_j + m_{j-1}$. It follows by induction that $\textrm{gcd}(m_j,m_{j+1}) =  \textrm{gcd}(m_0,m_{1}) = 1$, and thus any two successive $m_j$'s are coprime. In particular, given a prime number $p$ and an irrational number $t$, there are infinitely many integers $m,n$ such that $p \not | \;m$ and $|t-\frac{m}{n}|<\frac{1}{n^2}$. 

Fix a prime number $p$. We have infinitely many coprime numbers $m,n$ such that $|p\theta -\frac{m}{n}|<\frac{1}{n^2}$, and $p \not | \;m$, i.e.\ $|\theta -\frac{m}{pn}|<\frac{1}{pn^2}$ and $(m,pn) = 1$.

Identifying the circle with the reals mod 1, we set 
$$
f_0(x) = \left \{ \begin{matrix}
		2npx & \mbox{ if } & 0 \leq x \leq \frac{1}{2np} \\
		2-2npx & \mbox{ if } & \frac{1}{2np} \leq x \leq \frac{1}{np} \\
		0 & \mbox{ if } & \frac{1}{np} \leq x \leq 1
		\end{matrix} \right .
$$
and we set $g_0(x) = f_0(x-1/2np)$.
Let $\gamma(f) = f(x-m/np)$ (rational rotation by $m/np$). Notice that $\gamma$ is $np$-periodic (and does not have a smaller period).
Set $f_j(x) = \gamma^j(f_0)$,  $g_j(x) = \gamma^j(g_0)$ for $j = 1,2,\ldots,np-1$. 
One easily verifies that $f_0+f_1+\ldots+f_{np-1} + g_0+g_1+\ldots+g_{np-1} = 1$ and that  $f_jf_k = g_jg_k = 0$ for $j \neq k$.

Notice furthermore that each $f_j$ and $g_j$ are Lipschitz with Lipschitz constant $2np$. For $j = 0,1,\ldots,p-1$, set 
$$
\tilde{f_j} = \sum_{k=0}^{n-1}f_{kp+j}   , \; \tilde{g_j} = \sum_{k=0}^{n-1}g_{kp+j}
$$
Notice that again, the $\tilde{f}_j$'s are pairwise orthogonal, the $\tilde{g}_j$'s are pairwise orthogonal, we have
$$
\tilde{f}_0+\tilde{f}_1+\ldots+\tilde{f}_{p-1}+\tilde{g}_0+\tilde{g}_1+\ldots+\tilde{g}_{p-1}=1,
$$
those functions all have Lipschitz constant $2np$, and 
$$
\gamma(\tilde{f}_j) = \tilde{f}_{j+1\, \textrm{(mod)}\,p}  , \; \gamma(\tilde{g}_j) = \tilde{g}_{j+1\, \textrm{(mod)}\,p}
$$
Now, we note that for all $x$ we have
$$
\left | \gamma(\tilde{f}_j)(x) - \alpha(\tilde{f}_j)(x) \right | 
= 
\left | \tilde{f}_j\left(x - \frac{m}{pn}\right) -   \tilde{f}_j(x - \theta) \right | \leq 2np \cdot \left |\theta - \frac{m}{pn}\right | \leq \frac{2np}{pn^2} = \frac{2}{n}
$$
and the same estimate holds for the $\tilde{g}_j$'s. By choosing $n$ sufficiently large, we have the required almost-permutation estimate.
The other requirements in the definition of  Rokhlin dimension $1$ are immediate. (And clearly, no action on a connected space can have Rokhlin dimension $0$.)
\end{proof}

\section*{Appendix: Crossed products, nuclear dimension and $\mathcal{Z}$-stability}

\setcounter{section}{1}
\setcounter{Thm}{0}

\renewcommand{\thesection}{\Alph{section}}
\label{section:preliminaries}

\begin{Not}
To fix notation we recall the construction of the reduced crossed product. Let $G$ be a discrete group, $\A \hookrightarrow B(H)$ 
a $C^*$-algebra acting faithfully on the  Hilbert space $H$ and $\alpha: G \to \textup{Aut}(\A)$ an action of $G$ on $\A$.
Define $\pi : \A \to B(\ell^2(G) \otimes H)$ by $$\pi (a) (e_g \otimes \xi)= e_g \otimes \alpha_{g^{-1}}(a)\xi$$ and 
$$\lambda_h ( e_g \otimes \xi)= e_{hg} \otimes \xi,$$
where $g,h \in G$, $a \in \A$ and $\xi \in H$. Then $\lambda_g \pi(a) = \pi ( \alpha_g (a)) \lambda_g$ i.e.\;$(\pi , \lambda)$ is a 
covariant representation. The reduced crossed product $\A \rtimes_{\alpha,r} G$ is the 
$C^*$-subalgebra of $B(\ell^2(G) \otimes H)$ generated by $\{ \pi (a)\lambda_g \mid a \in \A, g \in G\}$. This algebra does not depend on 
the choice of the faithful representation of $\A$. The groups we consider in  this paper are either finite or equal to $\Z$, hence amenable. 
For amenable groups the reduced crossed product $\A \rtimes_{\alpha,r} G$ coincides with the universal crossed product $\A \rtimes_{\alpha} G$. 
Using matrix units, $\pi (a)$ can also be written as 
$$
\pi(a)= \sum_{g \in G} e_{g,g} \otimes \alpha_{g^{-1}}(a)
$$
and
$$
\pi(a) \lambda_h = \sum_{g \in G} e_{g,h^{-1}g} \otimes \alpha_{g^{-1}}(a).
$$
Thus if $G$ is finite with $n$ elements then we may regard $\A \rtimes_{\alpha} G$ as a subalgebra of $M_n(\A)$ in a natural way and if
$G=\Z$  we will regard $\A \rtimes_{\alpha} \Z$ as a subalgebra of $B(\ell^2(\Z) \otimes H)$. 
We will usually  drop $\pi$ from the notation and denote $\lambda_g$ by $u_g$.
\end{Not}

We recall that a c.p.\ contraction $\varphi:\A \to \B$ is said to be an order zero map if whenever $x,y$ are positive elements in $\A$ such 
that $xy = 0$ then $\varphi(x)\varphi(y) = 0$.

Order zero maps play a central role in the definition of decomposition rank and nuclear dimension which we recall for the reader's convenience. (Cf.\ \cite{kirchberg-winter} 
and \cite{winter-zach}.)

\label{def-dr-nd} 
\begin{Def}
Let $\A$ be a $\mathrm{C}^{*}$-algebra, $\mathcal{F}  $  a finite-dimensional $\mathrm{C}^{*}$-algebra and $n \in \mathbb{N}$.
\begin{enumerate}
\item A c.p.\ map $\varphi : \mathcal{F}   \to \A$  is $n$-decomposable if there is a decomposition 
\[
\mathcal{F}  =\mathcal{F}  ^{(0)} \oplus \ldots \oplus \mathcal{F}  ^{(n)}
\]
such that the restriction $\varphi^{(i)}$ of $\varphi$ to $\mathcal{F}  ^{(i)}$ has order zero for each $i \in \{0, \ldots, n\}$. 
\item $\A$ has decomposition rank $n$, $\dr \A = n$, if $n$ is the least integer such that the following holds: For any finite subset $\F  \subset \A$ and $\varepsilon > 0$, 
there is a finite-dimensional c.p.c.\ approximation $(\mathcal{F}  , \psi, \varphi)$ for $\F$ with tolerance $\varepsilon$ (i.e., $\mathcal{F}  $ is finite-dimensional, 
$\psi: \A \to \mathcal{F}  $ and $\varphi:\mathcal{F}   \to \A$ are c.p.c.\ and $\|\varphi \psi (b) - b\| < \varepsilon \; \forall \, b \in \F$) 
such that $\varphi$ is $n$-decomposable. If no such $n$ exists, we write $\dr \A = \infty$.  
\item $\A$ has nuclear dimension $n$, $\nd \A = n$, if $n$ is the least integer such that the following holds: For any finite subset $\F  \subset \A$ and 
$\varepsilon > 0$, there is a finite-dimensional c.p.\  approximation $(\mathcal{F}  , \psi, \varphi)$ for $\F$ to within $\varepsilon$ (i.e., $\mathcal{F}  $ is finite-dimensional, 
$\psi: \A \to \mathcal{F}  $ and $\varphi:\mathcal{F}   \to \A$ are c.p.\  and $\|\varphi \psi (b) - b\| < \varepsilon \; \forall \, b \in \F$) such that $\psi$ is c.p.c., and $\varphi$ is 
$n$-decomposable with c.p.c.\ order zero components $\varphi^{(i)}$. If no such $n$ exists, we write $\nd \A = \infty$. 
\end{enumerate}
\end{Def}

Note that in (3) we may assume  $\varphi \circ \psi$ to be contractive by \cite[Remark~2.2(iv)]{winter-zach}. 

\begin{Def}
Let $\mathcal{F}  $ be a finite dimensional $C^*$-algebra. Let $\A$ be a $C^*$-algebra, and let $\delta>0$. A c.p.\  contraction $\varphi:\mathcal{F}   \to \A$ is a $\delta$-order zero map if for any positive contractions $x,y\in \mathcal{F}  $ such that $xy = 0$ we have that $\|\varphi(x)\varphi(y)\|\le\delta$.
\end{Def}

We recall that order zero maps from finite dimensional $C^*$-algebras have the following stability property (\cite[Proposition 2.5]{kirchberg-winter}).
Let $\mathcal{F}  $ be a finite dimensional $C^*$-algebra, then for any $\eps>0$ there is a $\delta >0$ (depending on $\mathcal{F}  $ and $\eps$) such that if $\A$ is any 
$C^*$-algebra and $\varphi:\mathcal{F}   \to \A$ is a $\delta$-order zero map then there is an order zero map $\varphi':\mathcal{F}   \to \A$ such that $\|\varphi - \varphi'\|<\eps$. 
Using this, we easily obtain the following technical Lemma.

\begin{Lemma}
\label{lemma:approx-nuc-dim}
Let $\A$ be a $C^*$-algebra. $\A$ has nuclear dimension at most $n$ if and only if for any finite set $\F \subseteq \A$ and any $\eps>0$ there exists a 
finite dimensional $C^*$-algebra $\mathcal{F}   = \mathcal{F}  ^{(0)} \oplus \ldots \oplus  \mathcal{F}  ^{(n)}$ such that for any $\delta>0$ there exist c.p.\  maps $\psi:\A \to \mathcal{F}  $, $\varphi:\mathcal{F}   \to \A$ 
such that $\psi$ is contractive, $\|\varphi \circ \psi (a) - a\|<\eps$ for all $a \in \F$ and $\varphi|_{\mathcal{F}^{(j)}}$ is a c.p.c.\ $\delta$-order zero map for any $j = 0,1,\ldots,n$.

$\A$ has decomposition rank at most $n$ if the same holds where furthermore $\varphi$ is assumed to be a contraction. 
\end{Lemma}

Recall from \cite{JiaSu:Z} and \cite{TomsWin:ZASH} that the Jiang-Su algebra $\mathcal{Z}$ can be written as an inductive limit of prime dimension drop intervals of the form
\[
\Zh_{p,p+1} = \{f \in C([0,1] , M_p\otimes M_{p+1}) \mid f(0) \in M_p \otimes 1 \; , \; f(1) \in 1 \otimes M_{p+1}\}
\]
and that a separable unital $C^{*}$-algebra $\mathcal{A}$ is $\mathcal{Z}$-stable if and only if it admits almost central  unital embeddings of $\mathcal{Z}_{p,p+1}$ for some $p\ge 2$. With the aid of \cite{rordam-winter} one can give the following useful alternative formulation, cf.\ \cite{winter-dr-Z}.

\begin{Prop}
\label{Z-stable-crit}
Let $\A$ be a separable unital $C^{*}$-algebra. $\A$ is $\mathcal{Z}$-stable if and only if for some $p \ge 2$ there are c.p.c.\ order zero maps 
\[
\Phi:M_{p} \to A_{\infty} \cap A' \mbox{ and } \Psi:M_{2} \to A_{\infty} \cap A'
\]
satisfying
\[
\Psi(e_{11}) = 1_{A_{\infty}} -\Phi(1_{M_{p}}) 
\]
and
\[
\Phi(e_{11}) \Psi(e_{22}) = \Psi(e_{22}) \Phi(e_{11}) = \Psi(e_{22}).
\]
\end{Prop}

%\newpage

%\bibliographystyle{amsplain}

%\bibliography{Rokhlin-nuclear-dimension1}

\providecommand{\bysame}{\leavevmode\hbox to3em{\hrulefill}\thinspace}
\providecommand{\MR}{\relax\ifhmode\unskip\space\fi MR }
\providecommand{\MRhref}[2]{%
  \href{http://www.ams.org/mathscinet-getitem?mr=#1}{#2}
}
\providecommand{\href}[2]{#2}

\end{document}